%% file: main.tex
\documentclass[hidelinks,onefignum,onetabnum]{siamart220329}

\input{ex_shared}

\input{style}

\ifpdf
\hypersetup{
  pdftitle={A 3D Machine Learning based VOF scheme without explicit interface reconstruction},
  pdfauthor={Moreno Pintore and Bruno Despr\'es}
}
\fi

\usepackage{hyperref}
\usepackage{multirow}

\begin{document}

\maketitle

\begin{abstract}
We present a machine-learning based Volume Of Fluid method to simulate multi-material flows on three-dimensional domains. One of the novelties of the method is that the flux fraction is computed by evaluating a previously trained neural network and without explicitly reconstructing any local interface approximating the exact one. 
The network is trained on a purely synthetic dataset generated by randomly sampling numerous local interfaces and which can be adapted to improve the scheme on less regular interfaces when needed. Several strategies to ensure the efficiency of the method and the satisfaction of physical constraints and properties are suggested and formalized. Numerical results on the advection equation are provided to show the performance of the method.
We observe numerical convergence as  the size of the mesh tends to zero  $h=1/N_h\searrow 0$, with a better rate than two reference schemes.
\end{abstract}



\section{Introduction}
Multi-material or multi-phase numerical simulations are crucial in numerous engineering applications, a non-exhaustive list of examples includes water waves, nuclear reactors, jets and sprays, oil or gas pipelines, thermal power stations, and mold filling applications. 
Due to the presence of sharp interfaces between the involved materials, performing accurate simulations is complex and requires ad-hoc techniques. Indeed, a numerical scheme with too much diffusion would not be able to capture the sharp and discontinuous nature of the interface, whereas a scheme with less diffusion might not be stable enough to properly simulate the evolution of the phenomenon of interest.
In this context, the Volume Of Fluid (VOF) method is one of the most commonly used techniques when exact conservation of one or more physical quantities is needed. Even though the first papers on the VOF method are not recent, see  \cite{hirt1981volume} and later references, the research field is still very active  because of the importance of the applications and the complexity of the problem, which further increases for three-dimensional problems involving more than two materials. See, for example, \cite{Mengi,Tavi,Hani,Zhaoi} and references therein. VOF methods are largely based on geometrical heuristic algorithms, even though a recent work establishes
error estimates in some cases \cite{cohen2025high}.
The essence of most VOF methods we know about (see our reference list)  is to reconstruct some interface parameters from cell averages, and then to use this information to design  numerical fluxes in a Finite Volume scheme.

Our work in three dimensions  is based on the coupled VOF and Machine Learning  methodology (coined as VOFML) for two-dimensional interfaces  proposed in 
\cite{despres2020machine,ancellin2023extension}. Its main feature, which is also  its main originality,
 is that  we do not  reconstruct interface parameters. Instead, we reconstruct the flux directly from the cell averages.
The  contribution of this work is to provide evidence that such a 3D VOFML is efficient.
Since we do not  
explicitly reconstruct three-dimensional 
interface parameters,
 it results in a tremendous conceptual simplification 
with respect to most of the literature on VOF schemes with or without Machine Learning, for which we refer the reader to  
\cite{hirt1981volume,fakhreddine2024directionally,pilliod2004second,weymouth2010conservative,dyadechko2008reconstruction,kawano2016simple,pathak2016three,cahaly2024plic} and references therein.
 In our opinion, this conceptual simplification  is an example of the kind of new algorithms that Machine Learning can contribute to generate.

The general context of VOF methods is easily explained  on our model problem.
It is  a practical simple scenario with two incompressible immiscible fluids, fluid A and fluid B, in the space-time domain $\Omega\times[0,T]$, with $\Omega\in\R^3$ and $T>0$. Let $\alpha:\Omega\times[0,T]\rightarrow\R$ be the characteristic function of the fluid A, i.e. $\alpha(x,y,z,t)=1$ if the position $(x,y,z)\in\Omega$ is occupied by the fluid A at time $t\in[0,T]$, and  $\alpha(x,y,z,t)=0$ otherwise. Because of the immiscible and incompressible nature of the fluids, the conservation of mass implies the following conservation equation for $\alpha$:
\begin{equation}\label{eq:prob}
\partial_t \alpha + \nabla\cdot(\alpha u) = \partial_t \alpha + u\cdot \nabla\alpha u= 0,
\end{equation}
where $u$ is a divergence-free velocity field  
$
\nabla \cdot u=0$. 
In order to numerically solve \eqref{eq:prob}, one introduces a mesh on $\Omega$ and associates with each cell $C$ of such a mesh the volume fraction 
\[
\alpha_C = \dfrac{1}{|C|}\int_C \alpha\, dx\,dy\,dz,
\]
where $|C|$ is the volume of the cell $C$. Note that $\alpha_C=1$ if $C$ contains only fluid A, $\alpha_C=0$ if it contains only fluid B, and $\alpha_C\in(0,1)$ if it contains both fluids. The idea of the VOF method is to use the volume fractions available at time $t$ to compute the volume fractions at time $t+\Delta t$ without explicitly knowing the distribution of the two fluids, and to iterate this procedure to cover the entire time interval $[0,T]$ to mimic the evolution of the unknown function $\alpha$. The incompressibility hypothesis can be removed for real  applications.

In order to compute the fluxes required to compute the volume fractions at the subsequent time step, most VOF schemes explicitly reconstruct, for each cell $C$, a simple interface that approximates the exact unknown interface between the two fluids. This simple interface is computed using the volume fractions of the cell $C$ and its neighbor cells (such a set of cells is named stencil), and it is often represented by affine or quadratic functions of $\x$. In particular, the most popular class of VOF methods adopts the Piecewise Linear Interface Calculation (PLIC) approach, where the interface is represented by a planar polygon in each cell. The most famous PLIC representatives are the Youngs method \cite{youngs1982time}, where the plane normal ${\mathbf n}_\alpha$ is computed by using the relation ${\mathbf n}_\alpha = \frac{\nabla \alpha}{|\nabla\alpha|}$ and its location is computed by matching the the volume fraction inside $C$, the LVIRA method (Least-square Volume-of-fluid Interface Reconstruction Algorithm) and its efficient variant ELVIRA (Efficient LVIRA) \cite{pilliod2004second}, where the planar interface is computed by solving a least squares minimization problem in each cell, and the height function method \cite{weymouth2010conservative}, where the interface normal is efficiently computed using a height-function technique. 
Because of the longevity of VOF methods, several alternatives exist. For example, it is possible to avoid the use of a stencil by introducing additional variables in each cell that represent the position of the center of mass of the fluid inside the cell, as in the Moment Of Fluid method \cite{dyadechko2008reconstruction}, or to bypass the interface reconstruction by using explicit algebraic formulas \cite{despres2001contact} to directly compute the required fluxes.

When considering three-dimensional simulations, however, even the implementation of standard PLIC methods on Cartesian meshes is complex and it is often advisable to introduce further approximations \cite{kawano2016simple} to simplify the implementation and reduce the associated computational cost. The situation is even more complex when more than two fluids are present in the simulation, because simple interfaces are not expressive enough to properly represent the complex distributions of more than two fluids in a single cell. For example, in \cite{caboussat2008numerical} an additional minimization procedure in each cell with 3 fluids is considered to localize the triple point, and in \cite{pathak2016three} a different minimization strategy is adopted, assuming that the fluid ordering is known a priori, to localize each interface. Such minimization strategies often rely on iterative clipping of polygons, which is computationally expensive in 3D. Therefore, it is often necessary to consider alternative and more complex implementations like the one proposed in \cite{kromer2023efficient} to reduce the computational overhead.
Because of all these sources of complexity and thanks to the rapidly increasing popularity of Scientific Machine Learning, in the last few years, numerous techniques based on neural networks \cite{goodi} have been proposed. For example, in \cite{cahaly2024plic} and in \cite{rushdi2025efficient} two PLIC variants in which the interface normal is computed through a neural network are proposed, in \cite{nakano2025machine} a graph neural network is used to describe the local interface in unstructured meshes, in \cite{qi2019computing} and in \cite{kraus2021deep} the neural network is involved only in the curvature computation, whereas in \cite{mak2025machine} the authors propose a neural network to directly predict the volume fractions at the subsequent time step without computing and using the fluxes.

In this manuscript, we focus 
on extending to  three-dimensional simulations the Volume Of Fluid - Machine Learning (VOF-ML) method proposed in \cite{despres2020machine}, which is an alternative VOF method relying on a neural network. As discussed in \cite{despres2020machine} for two-dimensional problems, the main idea of VOF-ML is to train a neural network to predict, given the volume fractions in a suitable stencil, the flux that has to be used in a standard Finite Volume (FV) scheme \cite{leveque2002finite}. Moreover, no simulations or physical measurements are required to train the network, because the training dataset is cheaply assembled from geometrical considerations. In fact, the training dataset contains several distributions of pairs of fluids in a stencil and the corresponding fluxes, and this information does not depend on the equation of interest, the type or number of materials, or the domain $\Omega$. All this additional information is only used by the FV scheme that uses the flux predicted by the neural network. For example, even if the network is trained on pairs of fluids, simulations with more than two fluids can be performed as discussed in \cite{ancellin2023extension} without retraining the neural network on distributions including more materials. The manuscript contains the description of the method and a few implementation details that can be used to enforce physical properties and more accurately compute the fluxes.

To summarize, we provide a list of features that characterize the proposed method:
\begin{itemize}
\item it addresses  3D configurations;
\item it exploits the   capabilities and nonlinearity of deep neural networks; 
\item  in the training phase it   exploits the interface information, but then in the inference phase   it computes the fluxes without explicitly reconstructing the interface;
\item it easily handles distributions where the interface is not regular, which is particularly important in presence of more than two materials;
\item its implementation is easier than standard three-dimensional VOF methods and its efficiency relies on highly efficient deep learning libraries (e.g. Tensorflow \cite{tensorflow}, PyTorch \cite{pytorch} or JAX \cite{jax}). 
\end{itemize}

The manuscript is organized as follows. In Section \ref{sec:method} we describe the basic concepts of the method, with particular focus on the neural network architecture, on the training dataset generation, on the network training and on its usage inside a FV solver. More advanced strategies, introduced to improve the efficiency and accuracy of the scheme and to enforce physical constraints, are discussed in Section \ref{sec:strategies}. Numerical results are shown in Section \ref{sec:numerical_results} to empirically validate the method, whereas the final conclusions are presented in Section \ref{sec:conclusion}.

\section{Method description}\label{sec:method}

As discussed in the Introduction, let us consider equation \eqref{eq:prob} on the space-time domain $\Omega\times[0,T]$. We consider a uniform Cartesian mesh ${\cal T}_h$ on the cubic
domain $\Omega=[0,L]^3$ comprising $N_h$ cells in each direction. Generalization to Cartesian meshes with elongated cells or with a different number of cells in each direction is straightforward. Using a standard notation, we denote by $C_{i,j,k}$, $i,j,k=1,\dots,N_h$, a generic cell of ${\cal T}_h$ 
$$
C_{i,j,k}= \left\{
(x,y,z)\in \Omega :  ih < x <(i+1)h, \  jh < y <(j+1)h, \
 kh < z <(k+1)h.
\right\}
$$
 Moreover, we also introduce the stencil $S_{i,j,k}^m$ of margin $m\in\N$ and centered in $C_{i,j,k}$ as the indices   of the $m^3$ cells around $C_{i,j,k}$
\[
S_{i,j,k}^m = \{ (i',j',k') :\,\,\, |i-i'|\le m,\,\,\, |j-j'|\le m,\,\,\, |k-k'|\le m\}.
\]
The stencil has a central role in any VOF scheme, because the volume fractions of its cells are used to compute the flux needed to compute the volume fractions at the next time instant. A two-dimensional schematic representation of the VOF method is shown in Figure \ref{fig:vof_idea}. 

\begin{figure}[!ht]
\centering
\begin{subfigure}{0.49\textwidth}
\centering
    \includegraphics[width=0.7\textwidth, clip]{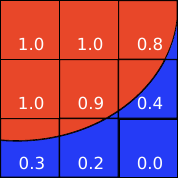}
    \caption{Example of volume fractions.}
\end{subfigure}
\begin{subfigure}{0.49\textwidth}
\centering
    \includegraphics[width=0.7\textwidth, clip]{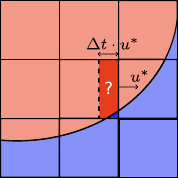}
    \caption{Representation of the unknown flux.}
    \label{fig:vof_idea_flux}
\end{subfigure}
\caption{Example of volume fractions and flux for a given distribution. The fluid A occupies the red region. The velocity in the central cell is $u^*$, the time-step length is $\Delta t$.}
\label{fig:vof_idea}
\end{figure}

The basic idea of the VOF-ML method is to exploit the nonlinear approximation capabilities of neural networks to capture the complex relationship between the volume fractions in each cell of a given stencil $S_{i,j,k}^m$, the velocity of the fluid inside the cell $C_{i,j,k}$, and the flux exiting from the cell $C_{i,j,k}$. Such a relationship is the central part of a VOF method, since the volume fractions and the velocity fields are known, but the flux is unknown. The flux reconstruction is then used in a standard FV scheme to obtain the new volume fractions.

The VOF-ML method is characterized by three main components:

\begin{itemize}
\item the generation of the dataset, used to train a neural network;
\item the training of the involved neural network;
\item the use of the trained neural network inside a standard FV scheme.
\end{itemize}

Each of these components is described in detail in the next subsections.

\subsection{A generic  neural network architecture}\label{sec:network}

We only consider fully-connected feed-forward neural networks. This is the simplest existing architecture, but it is very flexible and allows to accurately approximate very nonlinear and non-smooth operators. 

Let $L$ be the number of layers, let $A_\ell\in\mathbb{R}^{a_\ell\times a_{\ell-1}}$ and $b_\ell\in\mathbb{R}^{a_\ell}$, $\ell=1,...,L$, be $L$ matrices and vectors of known dimension and let $\rho:\R\rightarrow\R$ be a fixed nonlinear function, which can be applied element-wise to define the operator $\rho:\R^{a_\ell}\rightarrow\R^{a_\ell}$, $\ell=1,...,L$.
We define the parametric function $\alpha^\NN:\R^{a_0}\rightarrow\R^{a_L}$ associated with the proposed neural network as
\begin{equation} \label{eq:nn_formula}
  \begin{aligned}
  &x_0=\boldsymbol{x}, \\
 &x_\ell = \rho(A_\ell x_{\ell-1} + b_\ell), \hspace{1.cm} \ell = 1,...,L-1, \\
 &\alpha^\NN(\boldsymbol{x}) = A_{L} x_{L-1} + b_L.
  \end{aligned}
\end{equation} 

The input of the neural network is a vector containing the volume fractions in a stencil and the normalized velocity in a specific direction, whereas its output represents the approximation of the flux in the chosen direction. Therefore, the input dimension $a_0$ and the output dimension $a_L$ have to be $a_0 = m^3 + 1$ and $a_L=1$.

Moreover, the nonlinear function $\rho$, also named activation function \cite{dubey2022activation}, can be arbitrarily chosen. Some common choices are $\rho(x)=\text{ReLU}(x)=\max(0,x)$, $\rho(x)=\text{RePU}(x)=\max\{0,x^p\}$ for $1<p\in\N$, $\rho(x)=1/(1+e^{-x})$, $\rho(x)=\tanh(x)$ and $\rho(x)=\log(1+e^x)$. In this work, we only consider the Rectified Linear Unit ($\rho(x)=\text{ReLU}(x)=\max(0,x)$) because many FV flux limiters can be exactly expressed by using such a function. For example, the operator
\begin{equation*}
\text{minmod}(a,b) = \left\{
\begin{aligned}
&0 && ab\le 0\\
\min(|a|, |b|)\, \text{sign}(&a) && ab > 0
\end{aligned}
\right.
\end{equation*}
can be expressed as $\text{minmod}(a, b) = R(a - R(a - b)) - R(-a -R(b-a))$, when $R$ is the ReLU function.
The entries of all matrices and vectors in \eqref{eq:nn_formula} are named trainable weights of the neural network. The training of the neural network is the optimization procedure through which such weights are optimized to minimize a suitable cost function, named loss function.

\subsection{Dataset generation}\label{sec:dataset}

In this section, we describe the way in which we generate the training dataset. It consists of numerous input-output pairs, where each data pair is associated with a suitable distribution of the fluids A and B on a single stencil.

\begin{itemize}
\item
The input is a vector of size $m^3+1$. It contains the $m^3$ volume fractions associated with the given fluids distribution and the local Courant number in the positive $x$-direction.
The  local Courant number $\beta$ is sampled  in an interval $(C_{\rm min},C_{\rm max})\in [0,1]$.
\item
The output is the numerical flux in the   $x$-direction for the model equation (\ref{eq:prob}).
\end{itemize}

Every input-output pair is obtained by purely geometrical considerations. 
In most papers combining neural networks and VOF schemes
\cite{weymouth2010conservative,kawano2016simple,cahaly2024plic}, the training dataset is generated using regular shapes or functions. Here, instead, we are interested in training a neural network that can accurately predict the flux in the presence of both smooth and non-smooth interfaces, because in many applications both cases are present in the same simulation, in different regions of the domain.
A generic case with more fluids can be solved by coupling 2-fluids sub-problems as discussed in \cite{ancellin2023extension}.

A  generic 2D example is in Figure \ref{fig:vof_idea}, where we subdivide the reference stencil in two regions A and B.
However, in 3D the situation is more complicated as sketched in the next  Figures.
In this work, we consider the following synthetic distributions, where the region A is:
\begin{itemize}
\item either the intersection between a  cube and one, two, or three parametric half-spaces,
\item or the intersection between a cube and the interior of a parametric ellipsoid.
\end{itemize}
All technical formulas are provided in Appendix \ref{sec:params}.

In Figure  \ref{fig:1planes} we show examples of a region cut by one plane,
in Figure  \ref{fig:2planes} we show examples of a region cut by two planes,
in Figure  \ref{fig:3planes} we show examples of a region cut by three planes, and, in Figure \ref{fig:ellips}, we show examples of portions of the surface of the cutting ellipsoids, discretized by a finite number of points.

\begin{figure}[!ht]
\centering
\begin{subfigure}{0.24\textwidth}
\includegraphics[width=1\textwidth, clip]{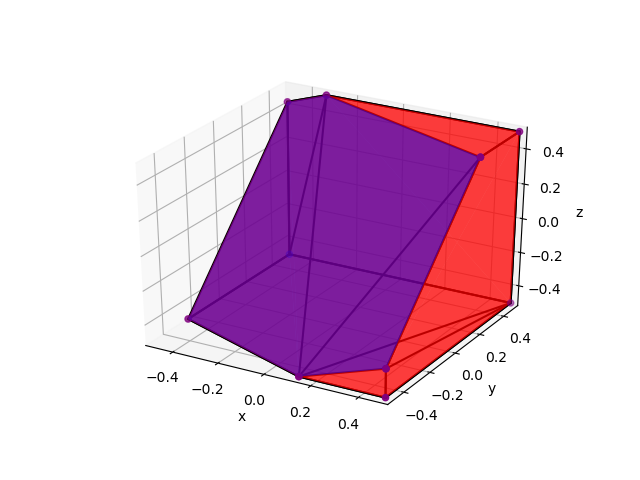}
\end{subfigure}
\begin{subfigure}{0.24\textwidth}
\includegraphics[width=1\textwidth, clip]{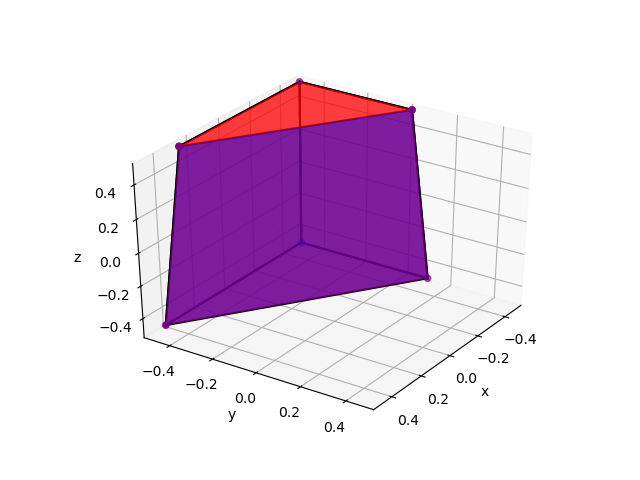}
\end{subfigure}
\begin{subfigure}{0.24\textwidth}
\includegraphics[width=1\textwidth, clip]{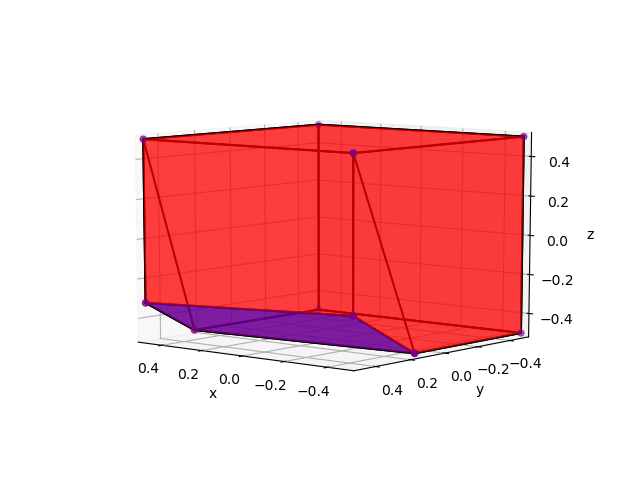}
\end{subfigure}
\begin{subfigure}{0.24\textwidth}
\includegraphics[width=1\textwidth, clip]{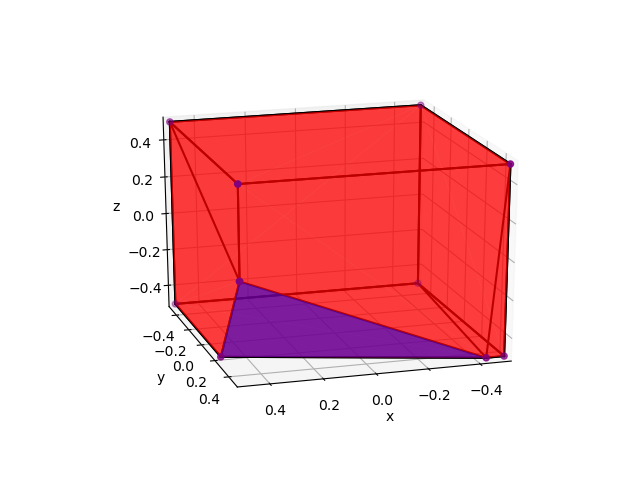}
\end{subfigure}
\caption{Examples of regions defined by a single half-space.}
\label{fig:1planes}
\end{figure}

\begin{figure}[!ht]
\centering
\begin{subfigure}{0.24\textwidth}
\includegraphics[width=1\textwidth, clip]{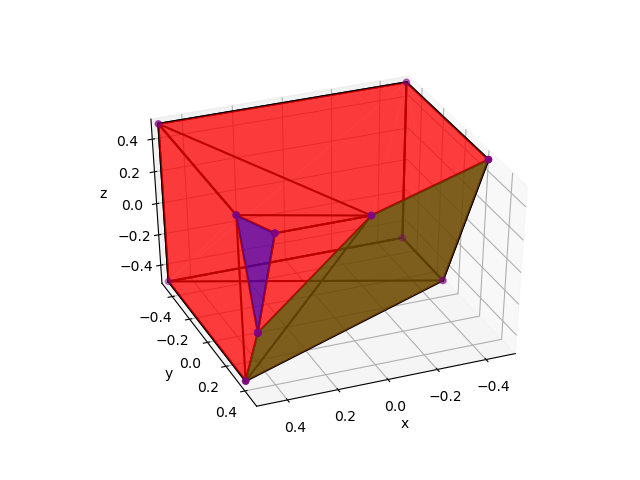}
\end{subfigure}
\begin{subfigure}{0.24\textwidth}
\includegraphics[width=1\textwidth, clip]{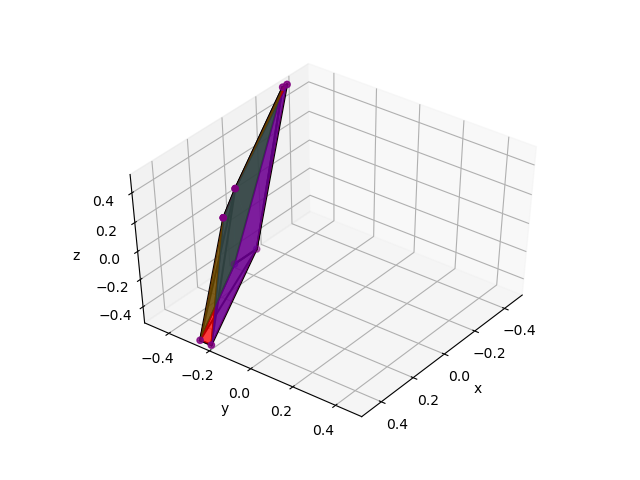}
\end{subfigure}
\begin{subfigure}{0.24\textwidth}
\includegraphics[width=1\textwidth, clip]{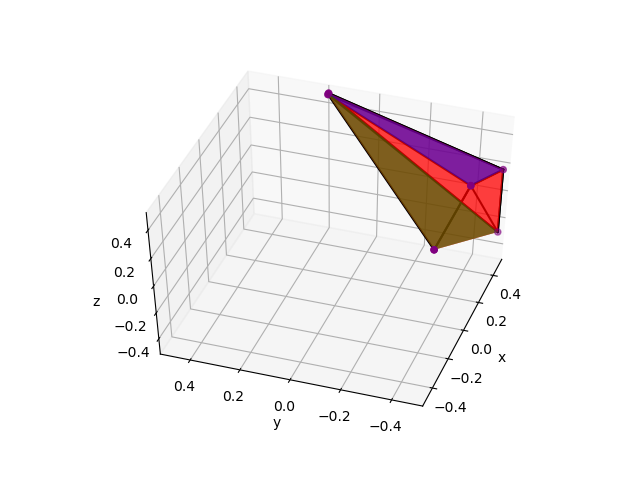}
\end{subfigure}
\begin{subfigure}{0.24\textwidth}
\includegraphics[width=1\textwidth, clip]{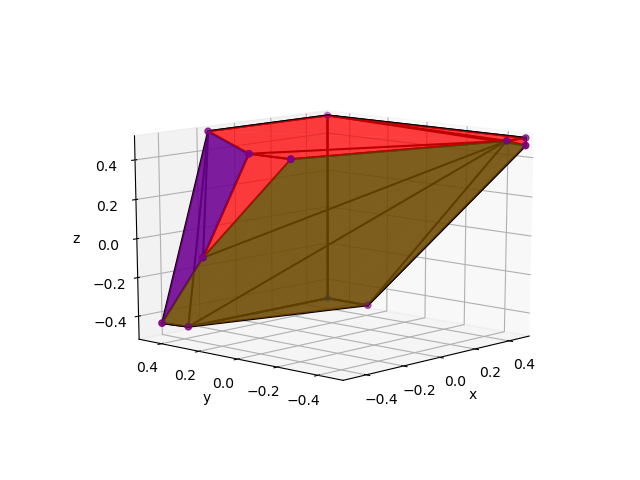}
\end{subfigure}
\caption{Examples of regions defined by two half-spaces.}
\label{fig:2planes}
\end{figure}

\begin{figure}[!ht]
\centering
\begin{subfigure}{0.24\textwidth}
\includegraphics[width=1\textwidth, clip]{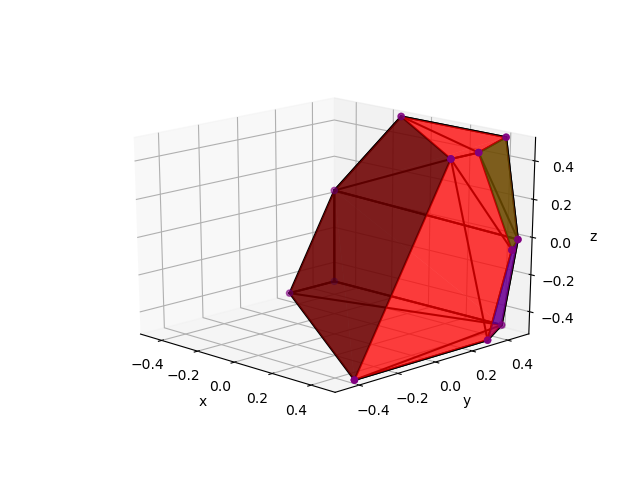}
\end{subfigure}
\begin{subfigure}{0.24\textwidth}
\includegraphics[width=1\textwidth, clip]{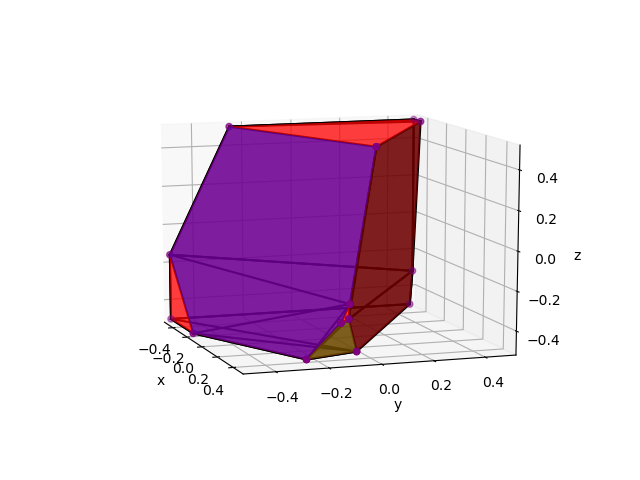}
\end{subfigure}
\begin{subfigure}{0.24\textwidth}
\includegraphics[width=1\textwidth, clip]{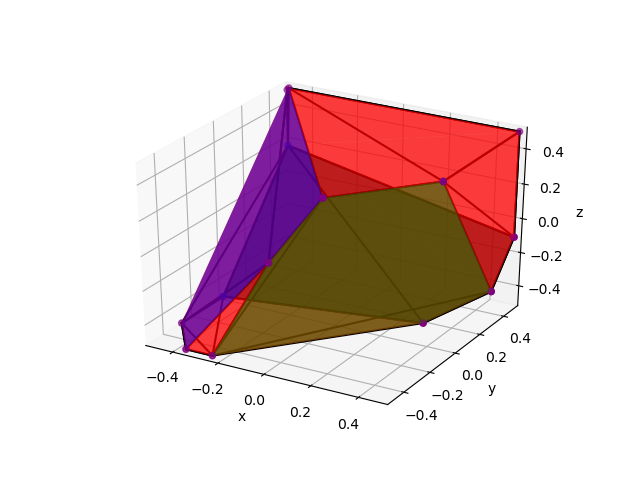}
\end{subfigure}
\begin{subfigure}{0.24\textwidth}
\includegraphics[width=1\textwidth, clip]{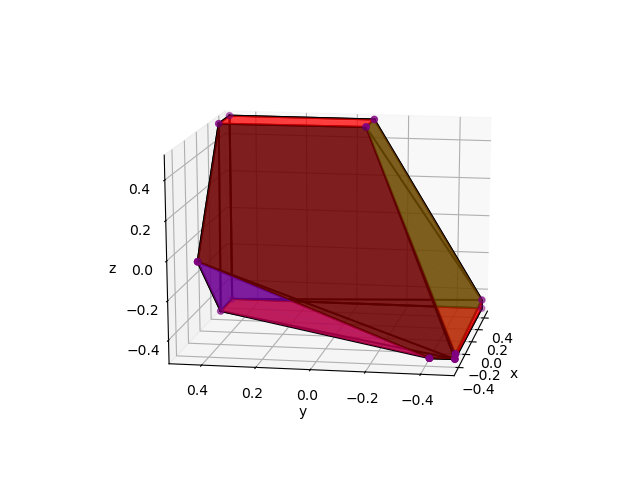}
\end{subfigure}
\caption{Examples of regions defined by three half-spaces.}
\label{fig:3planes}
\end{figure}

\begin{figure}[!ht]
\centering
\begin{subfigure}{0.24\textwidth}
\includegraphics[width=1\textwidth, clip]{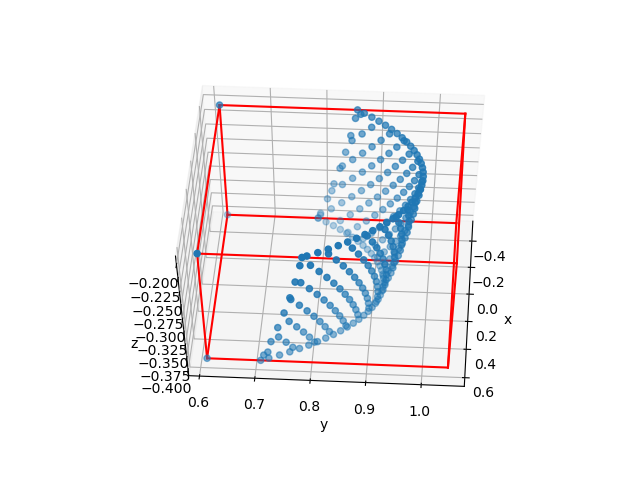}
\end{subfigure}
\begin{subfigure}{0.24\textwidth}
\includegraphics[width=1\textwidth, clip]{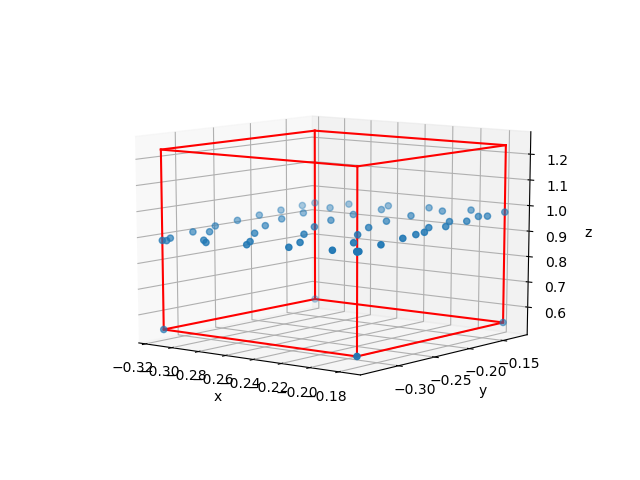}
\end{subfigure}
\begin{subfigure}{0.24\textwidth}
\includegraphics[width=1\textwidth, clip]{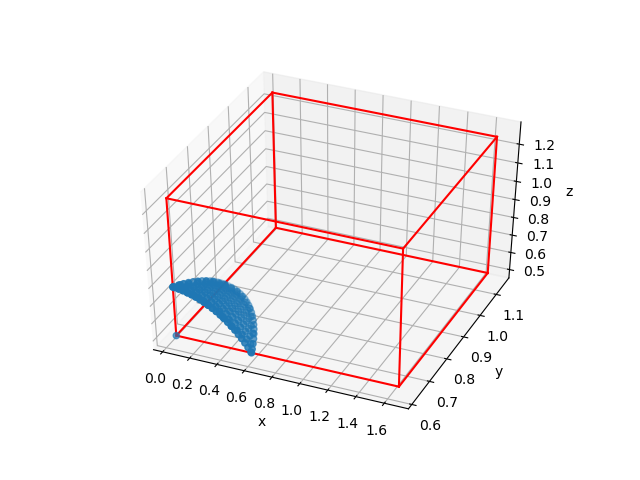}
\end{subfigure}
\begin{subfigure}{0.24\textwidth}
\includegraphics[width=1\textwidth, clip]{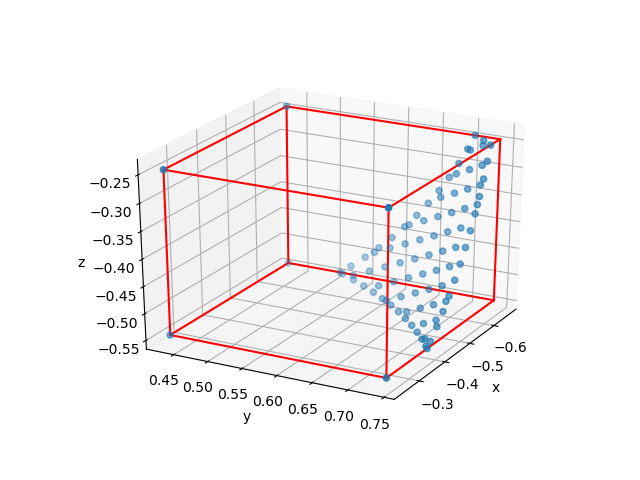}
\end{subfigure}
\caption{Examples of smooth regions cut by an ellipsoid generated as in Figure \ref{fig:fibonacci}.}
\label{fig:ellips}
\end{figure}

These regions are convex, but concave regions can be obtained by switching the roles of the two fluids.
Note that this is a continuous set of geometrical configurations, in the sense that the intersection of two or three half-spaces and the interior part of an ellipsoid can tend to a single half-space, which is the simplest considered case. Considering in our dataset the intersection between two or three half-spaces allows us to train the neural network on less regular distributions. We hope it results in
improving the prediction capabilities for specific applications, but without changing the nature of the VOF-ML approach.
\begin{remark}
The results of our numerical tests show  that this ensemble of configurations is sufficient to reach good prediction capabilities of our VOFML scheme in 3D.
\end{remark}

For each cell of the stencil, we have to compute the volume fraction associated with one of the two regions, say region  A. Performing such a computation by numerical quadrature would be very expensive or inaccurate, due to the discontinuous nature of the target function. Instead,  when considering one or more half-spaces, we explicitly construct the polytopes defined by the intersection between the region A and each cell of the stencil and we exactly calculate the volumes.
In the case of an ellipsoid, we start by approximating it with a polytope with $N_\text{pts-el}=10000$ vertices (generated as in Remark \ref{rem:fibonacci}),  then we construct the polytopes defined by the intersection of the approximated ellipsoid and each cell. This way, it is always possible to use known and efficient algorithms to evaluate the volume of a convex hull \cite{barber1996quickhull} and obtain the exact volume fractions in the first case, and a good approximation in the second one.

\begin{remark}[Discretization of the ellipsoid surface]\label{rem:fibonacci}
In this work, all the integrals involved in the dataset construction are computed through the convex hull algorithm. It is therefore necessary to discretize each ellipsoid with a polytope, which has to be an accurate approximation of the ellipsoid to avoid introducing large errors in the training dataset.

We construct the polytope with $N_\text{pts-el}$ vertices using the Fibonacci spiral \cite{keinert2015spherical}. To define such a set of points $\{x_i, y_i, z_i\}_{i=0}^{N_\text{pts-el}-1}$, let us consider the corresponding spherical coordinates $\{\phi_i, \theta_i\}_{i=0}^{N_\text{pts-el}-1}$. Denoting by $G_R$ the golden ratio $G_R=(\sqrt{5}+1)/2$, $\phi_i$ and $\theta_i$ can be computed as:
\[
\phi_i = 2\frac{  \pi i}{G_R}, \hspace{1cm} \theta_i = \cos^{-1}\left(1 - \frac{2 i + 1}{ n}\right),\hspace{1cm} i=0,\dots,N_\text{pts-el}-1.
\]
Then, the Cartesian coordinates $x_i$, $y_i$ and $z_i$ can be retrieved as:
\[
[x_i, y_i, z_i] = [\cos(\phi)\sin(\theta), \sin(\phi)\sin(\theta), \cos(\theta)],\hspace{1cm} i=0,\dots,N_\text{pts-el}-1.
\]
The obtained distribution is then mapped to the surface of the ellipsoid through a simple affine transformation. An example of the obtained distribution on the unit sphere, with $N_\text{pts-el}=1500$, is shown in Figure \ref{fig:fibonacci}. Note that the points have a quasi-uniform distribution on the entire surface.

\begin{figure}[!ht]
\centering
\includegraphics[width=0.35\columnwidth,keepaspectratio,clip]{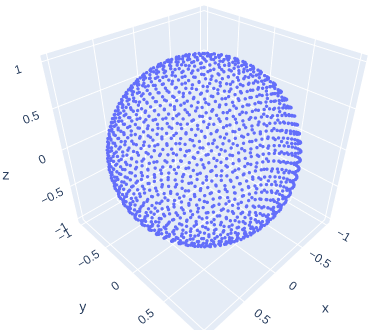}
\caption{Fibonacci spiral with $N_\text{pts-el}=1500$.}
\label{fig:fibonacci}
\end{figure}
\end{remark}

Due to the high dimension of the corresponding parameter spaces (see Appendix \ref{sec:params}), we sample the training distributions using a Latin Hypercube Sampling technique in order to weaken the effect of the curse of dimensionality.
At the end of this stage, one has designed the finite dataset of inputs
\begin{equation} \label{eq:dodo}
\mathcal D_{\rm inp}= \left\{
(\boldsymbol{x},\beta)\in
\mathbb R^{m^3} \times \mathbb R
\right\}, \qquad \# \mathcal D_{\rm inp}<\infty.
\end{equation}
The exact numerical flux is represented as a geometrical function $\alpha$ calculated as the volume fractions, i.e. using the convex hull algorithm
$$
\alpha : \mathbb R^{m^3+1} \times \mathbb R \to \mathbb R.
$$
The dataset of outputs is
$$
\mathcal D_{\rm out}=\left\{
\gamma\in
\mathbb R \mid
\gamma= \alpha(\boldsymbol{x},\beta)
\mbox{ for }(\boldsymbol{x},\beta)\in \mathcal D_{\rm inp}
\right\}
\subset \mathbb R.
$$
The total or training dataset is
$
\mathcal D=
\left\{
(\boldsymbol{x},\beta,\gamma)\in
\mathbb R^{m^3+2} \mid
(\boldsymbol{x},\beta)\in \mathcal D_{\rm inp} , \ \gamma= \alpha(\boldsymbol{x},\beta)
\right\}$.

\subsection{Training}\label{sec:training}

After the generation of the training dataset $\cal D$, a large set of $\#{\cal D}$ input-output pairs is available, and a supervised training procedure is performed. Let us consider the neural network described in Section \ref{sec:network} and fix the architecture free hyper-parameters $L, a_1, a_2, \dots, a_{L-1}$. 
The objective of training is to find a set of weights of the neural network such that $\alpha^\NN(\xx,\beta) \approx \alpha(\xx,\beta) $ for any possible pair $(\boldsymbol{x},\beta) \in \mathcal D$.

To do so, we construct a loss function that penalizes large differences between $\alpha^\NN$ and $\alpha$. 
We consider the standard Mean Squared Error (MSE) loss function $\cal L$ defined as:
\begin{equation}\label{eq:loss}
{\cal L} =\frac{1}{\#{\cal D_{\rm inp} }} \sum_{(\boldsymbol{x},\beta)\in{\cal D_{\rm inp} }} \left( \alpha^\NN(\xx,\beta) - \alpha(\xx,\beta)\right)^2.
\end{equation}
Note that, given a neural network architecture, the loss function $\cal L$ depends only on the dataset $\cal D$ and on the neural network's trainable weights. In particular, the neural network and its training do not depend on the PDE that has to be solved (in this case equation \eqref{eq:prob}) or on any previously computed expensive simulation. All our training data are synthetic.

Because of the intrinsic nonlinearity of the neural network representation,
the cost function  $\cal L$ is highly non-convex, and it is not possible to find global minima of such a cost function. Therefore, we content ourselves with finding local minima through an iterative descent algorithm. Numerous algorithms are available, in this work we use a combination of the ADAM optimizer \cite{kingma2014adam} and the BFGS optimizer \cite{wright1999numerical}.
At the end of the training, we have a function $\alpha^\NN$ that approximately represents the relation between the input vector  and the target
for all pair  $(\boldsymbol{x}, \beta)\in{\cal D}$, that is
\begin{equation} \label{eq:didi}
\alpha^\NN(\xx,\beta) \approx \alpha(\xx,\beta)  \mbox{ for all }(\boldsymbol{x},\beta) \in \mathcal D_{\rm inp}
.
\end{equation}

\subsection{Definition of the 3D VOFML scheme}\label{sec:nn_in_FV}

Let us consider Problem \eqref{eq:prob} on the domain $\Omega$, discretized by the mesh ${\cal T}_h$ with mesh size $\Delta x$. Let $\Delta t$ be the time step size. The initial condition is evolved in time one time step at a time until the entire time interval $[0,T]$ is covered.

Since we adopt a FV scheme, the volume fraction 
$\alpha_{i,j,k}$ at a given time $t$ is updated at the time $t+\Delta t$ as  $\overline{ \alpha_{i,j,k}}$ using the generic formula 
\begin{equation}\label{eq:fv_step}
\overline{ \alpha_{i,j,k}}= \alpha_{i,j,k} -  \sum_{f\in \mathcal F( C_{i,j,k})} \widetilde g(f,n)
\end{equation}
where the sum is extended over  the set of faces  $\mathcal F( C_{i,j,k})$ 
$$
\mathcal F( C_{i,j,k})=
F_{i,j,k}^x\bigcup F_{i,j,k}^y\bigcup F_{i,j,k}^z,
\quad
F_{i,j,k}^x=\left\{
f_{i\pm \frac12, j,k}
\right\}, \ F_{i,j,k}^y=\left\{
f_{i,j\pm \frac12,k}
\right\}, \
F_{i,j,k}^z=\left\{
f_{i,j,k\pm \frac12}
\right\}
,
$$
and $\widetilde g(f,n)$ is the signed flux across the face $f$ in the direction normal $n$ in the time interval $( t, t+\Delta t)$, and  $f_{i\pm \frac12, j,k}$, $f_{i, j\pm\frac12,k}$ or $ f_{i, j,k\pm\frac12}$
refer to the faces of the structured mesh. 
In practice, we use a directional splitting scheme to reduce the   3-dimensional problem (\ref{eq:fv_step}) to a triplet of 1-dimensional problems.
For the simplicity of the presentation only, we assume that the velocity field is constant with non-negative components in all directions
$$
u=(u_x, u_y, u_z) \qquad \mbox{ with } u_x, \ u_y, \ u_z \geq 0
$$
Then the splitting method is made of the following  3 steps
\begin{equation}\label{eq:fv_split}
\begin{aligned}
\alpha_{i,j,k}^{x} &= \alpha_{i,j,k}  - 
u_x \frac{    \Delta t }{ \Delta x}
\left(\alpha_{{i+ \frac12,j,k}} - \alpha_{{i- \frac12 ,j,k}} \right) , \\
\alpha_{i,j,k}^{y} &= \alpha_{i,j,k}^{x} - u_y  \frac{    \Delta t }{ \Delta x}
\left(
\alpha_{{i,j+ \frac12, k}} - \alpha_{{i,j- \frac12, k}}
\right) , \\
\overline {\alpha_{i,j,k} }&= \alpha_{i,j,k}^{y} - u_z  \frac{    \Delta t }{ \Delta x}
\left( 
\alpha_{{i, j, k+ \frac12}} - \alpha_{{i, j, k-  \frac12 }}
\right) .\\
\end{aligned}
\end{equation}
For stability reasons, we assume a natural CFL constraint $\max(u_x, u_y, u_z) \frac{ \Delta t} {\Delta x}\leq 1$.

The fluxes $\alpha_{{i+  \frac12,j,k}} $ and $ \alpha_{{i- \frac12 ,j,k}} $
are the fluxes going outside and inside on the two faces in $F_{i,j,k}^x$.
Similar reasonings hold for the other directions $y$ and $z$. The iterations of equation \eqref{eq:fv_split} allows to simulate the entire flow from time $t=0$ to the final time $t=T$.

\begin{definition}[Flux of the  VOFML scheme]
The originality of the VOFML scheme in 3D relies on the following choice of the numerical flux, which is hereafter given for the flux in the $x$ direction
for all $i,j,k$
\begin{equation} \label{eq:dudu}
\alpha_{{i+ \frac12,j,k}} = \alpha^{\mathcal N \mathcal N}(\mathbf x,\beta_x)
\end{equation}
where
\begin{itemize}
\item the volume fractions are taken locally 
$
\mathbf x =
\left(\alpha_{i', j',k'} \right)_{(i',j',k')\in S_{i,j,k}^m}$,
\item
$\beta_x= u_x \frac{    \Delta t }{ \Delta x}\geq 0$ is the Courant number in the $x$ direction,
\item  $\alpha^{\mathcal N \mathcal N}$ is the function trained in (\ref{eq:didi}) for the dataset $\mathcal D$ (\ref{eq:dodo}).
\end{itemize}
The same definition holds with evident modification in the $y$ and $z$ directions. If some components of the velocity are negative,  the stencil $ S_{i,j,k}^m $ is upwinded accordingly, and the Courant number is kept non-negative.
\end{definition}

Note that \eqref{eq:fv_split} is the general update rule used in many VOF schemes, the only difference is in the way the fluxes are computed.
In more classical VOF schemes such as \cite{hirt1981volume,fakhreddine2024directionally,pilliod2004second,weymouth2010conservative,dyadechko2008reconstruction,kawano2016simple,pathak2016three,cahaly2024plic},
the local stencil is used to reconstruct interface parameters, which are post-processed to evaluate the numerical flux.
In this  VOFML scheme the numerical flux is directly computed by evaluating a previously trained neural network.  It saves one step by comparison with the general VOF procedure
which results in an important algorithmic simplification.

\section{Physical constraints}\label{sec:strategies}



The VOFML scheme defined in the previous Section  is a purely data-driven approach  which does not take into account the natural symmetries  attached to the model
transport equation \ref{eq:prob}.
This is a common problem for most data driven techniques used in combination with numerical solvers.
It is therefore advisable to enforce additional physical constraints or symmetries in the network or in its training, as described in \cite{garanger,mimi,despres2020machine} in a different context.
This enforcement can be performed {\it a priori} or {\it a posteriori}.
What we call an  {\it a priori} approach in the context of ML schemes is a ML architecture which naturally provides a quantity satisfying the symmetries.
What we call an {\it a posteriori} approach is based on a post-processing.
In this Section, we propose and analyze a few possible  {\it a priori} strategies to increase the symmetry properties of the VOF-ML scheme.
We pay particular attention to the symmetry group of a cubic structure in 3D. An {\it a posteriori} strategy to avoid creation of negative mass is also proposed in Section \ref{sec:non-neg}.

\subsection{3D symmetries}\label{sec:symms}

Several natural symmetries exist on Cartesian meshes.
An important fact  \cite{lomont2014applications} is that a  cube has a symmetry group of order 48 (the octahedral group),
that is is left invariant by  47 non-trivial symmetries/permutations generically denoted as $\sigma$  (the trivial symmetry is the identity).
The octohedral group is isomorphic to $ S_4\times \mathbb Z_2$, that is once again $24\times 2$ elements.
It is clearly
more symmetries than in 2D.

It explains that some care is needed to implement a symmetry-preserving method in our 3D neural network approach.
In the following, we provide an arbitrary list
of basic symmetries that are actually sub-groups of cubic symmetries.
The notation "(1+$c$)" refers to the fact that the sub-group is made of $1+c$ elements, where the first one is the identity and the other ones are non-trivial.
We consider the following basic transformations, which are ordered with respect to one particular direction,   arbitrarily chosen to the $x$-direction.
\begin{figure}

\begin{subfigure}{0.49\textwidth}
\centering
\includegraphics[width=0.7\textwidth, clip]{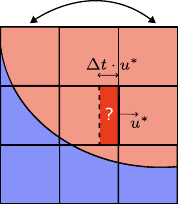}
\caption{$T_2$: the flux does not coincide.}
\label{fig:vof_idea_flux-1}
\end{subfigure}
\begin{subfigure}{0.49\textwidth}
\centering

\vspace{0.85cm}

\includegraphics[width=0.8\textwidth, clip]{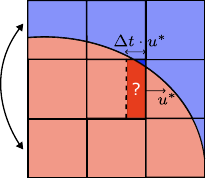}
\caption{$T_4$: the flux  coincides. }
\label{fig:vof_idea_flux-2}
\end{subfigure}

%
\caption{Comparison with the flux of  Figure \ref{fig:vof_idea} (on a 2D configuration).}
\label{fig:vof_idea_symms}
\end{figure}

\begin{itemize}
\item $T_1$: rotations leading to a reordering of the axes: the first axis, i.e. the one in the direction of the flux, can be the $x$-axis, the $y$-axis or the $z$-axis (1+2);
\item $T_2$: reflection with respect to the plane orthogonal to the first axis (1+1);
\item $T_3$: rotations of multiples of $\pi/2$ around the first axis (1+3);
\item $T_4$: reflection with respect to the plane orthogonal to the second axis (1+1).
\end{itemize}
All 48 symmetries of a cube can be obtained by combining these basic transformations. For example, a reflection around the third axis can be obtained by applying the transformations $T_3$ (with $\pi$ as the rotation angle) and
$T_4$. Note that each of the considered transformations maps the region occupied by the stencil to itself but permutes the cells composing it. Therefore, since in a VOF scheme one is only interested in the volume fractions, we always work with the corresponding permutations of the volume fractions, and we do not apply the geometrical transformations to the cube associated with the stencil.

Let us now consider an arbitrary distribution of two fluids inside a cubic stencil, and let us focus on the flux in the $x$-direction
for a  constant velocity of the form $(v,0,0)$.
Applying any transformation of type $T_1$ or $T_2$, then the  flux may vary as illustrated on the left of  Figure \ref{fig:vof_idea_symms}.
On the other hand, the flux remains the same whenever any transformation of type $T_3$ or $T_4$ is applied, as illustrated on the left of  Figure \ref{fig:vof_idea_symms}.
For this reason, we distinguish  two sub-groups of the symmetry group
$$
\left.
\begin{array}{ll}
\cal S_{\neq}= \left\{ \mbox{symmetries obtained by applying  transformations  $T_1$ and $T_2$} \right\},
&
\mbox{(exact flux may vary)},\\
\cal S_{=}= \left\{ \mbox{symmetries obtained by applying  transformations  $T_3$ and $T_4$} \right\},
&
\mbox{(exact flux is the same)}.
\end{array}
\right.
$$
Due to the different effects of these sets of transformations on the flux, we use the set $\cal S_=$ to impose physical properties to the neural network (see Section \ref{sec:constraints}) and the set $\cal S_{\neq}$ to cheaply enrich the training dataset.
In particular, note that to generate the triplet $(\xx,\beta,\gamma)\in{\mathcal D}$ the computation of $m^3+1$ integrals is required (the $m^3$ volume fractions included in the vector $\xx$ and the flux $\gamma = \alpha(\xx,\beta)$). On the other hand, given the untransformed triplet $(\xx,\beta,\gamma)\in{\mathcal D}$ and for any $\sigma\in{\cal S_{\neq}}$, the new triplet $(\sigma\xx,\beta,\gamma_\sigma)$ can be more cheaply computed by permuting the entries of the vector $\xx$ according to the transformation $\sigma$ and by computing just the integral $\gamma_\sigma= \alpha(\sigma\xx,\beta)$. We thus consider the following enriched input dataset:

\begin{equation} \label{eq:dodo2}
\mathcal D_{\rm inp}^\sigma= \left\{
(\sigma\boldsymbol{x},\beta)\in
\mathbb R^{m^3} \times \mathbb R, \,\sigma\in{\cal S_{\neq}}
\right\}, \qquad \# \mathcal D_{\rm inp}^\sigma<\infty,
\end{equation}
and the corresponding total dataset
\[
\mathcal D^\sigma=
\left\{
(\boldsymbol{x},\beta,\gamma)\in
\mathbb R^{m^3+2} \mid
(\boldsymbol{x},\beta)\in \mathcal D_{\rm inp}^\sigma , \ \gamma= \alpha(\boldsymbol{x},\beta)
\right\}.
\]

Note that $\mathcal D_{\rm inp} \subset\mathcal D_{\rm inp}^\sigma$ and $\mathcal D \subset\mathcal D^\sigma$ because the identity belongs to  ${\cal S_{\neq}}$. In the following, for the sake of simplicity, we drop the superscript $\sigma$ and we always refer to $\mathcal D_{\rm inp}^\sigma$ and $\mathcal D^\sigma$ as $\mathcal D_{\rm inp}$ and $\mathcal D$ because the method description is not influenced by the particular choice of the training dataset.

\subsubsection{Enforcing  symmetry constraints}\label{sec:constraints}

Let us consider a set of volume fractions in a stencil and a given velocity $u^*$ in its central cell. As discussed in Section \ref{sec:symms}, if a transformation in $\cal S_=$ is applied to the stencil, then the exact flux in the $x$-direction does not vary. As a consequence, the flux computed with any VOF scheme should not change with respect to the flux computed on the original stencil. However, even with a perfectly devised training, the neural network would not satisfy such an important physical constraint.
Therefore, we need to add such a constraint in our neural network architecture to obtain networks that satisfy it
by construction.
We thus want to modify the flux computation such that, if a transformation in $\cal S_=$ is applied to the stencil, the output of the neural network does not vary. 
There are 8 different  permutations that preserve the flux
$$
\mathcal S_{=}=\left\{
\sigma_1, \sigma_2, \dots, \sigma_8
\right\}.
$$
Given a
neural network $ \alpha^\NN$ as in (\ref{eq:didi}), we define a new neural network  $\widehat \alpha^\NN$ by
\begin{equation}\label{eq:net_perm}
\widehat \alpha^\NN(\mathbf x, \beta) = \frac 18 \sum_{j=1}^8 \alpha^\NN(\sigma_j \mathbf x, \beta)
\mbox{ for all }(\mathbf x, \beta)\in \mathbb R^{m^3} \times \mathbb R .
\end{equation}
This operation is performed by a linear combination, so it is an admissible neural network.
The cost of evaluating $\widehat \alpha^\NN$ is a priori 8 times the cost of evaluating $\alpha^\NN$,
in practice less because some saving is possible with vectorization of the calculations.


\begin{lemma}
\label{lem:lem1}
The new neural network preserves symmetries in $ \mathcal S_{=}$, that is
$$
\widehat \alpha^\NN(\sigma_j \mathbf x, \beta)= \widehat \alpha^\NN(\mathbf x, \beta)
\mbox{  for all } (\mathbf x, \beta)\in \mathbb R^{m^3} \times \mathbb R
\mbox{ and all }\sigma_j \in \mathcal S_{=}.
$$
\end{lemma}

\begin{proof}
This is due to the group structure of $\mathcal S_{=}$. Indeed
$$
\widehat \alpha^\NN(\sigma_j \mathbf x, \beta)= \frac 18 \sum_{k=1}^8 \alpha^\NN(\sigma_k \sigma _j  \mathbf x, \beta)=
\frac 18 \sum_{l=1}^8 \alpha^\NN(\sigma_l \mathbf x, \beta)
$$
because, for any $j=1,\dots, 8$, the index $l$ in $\sigma_l=\sigma_k \sigma _j $ covers $l=1,\dots, 8$ as $k=1,\dots, 8$.
\end{proof}
%
%

The same post-processing technique can be used to enforce additional symmetries.
For example, we can consider that the numbering of the fluid of interest leads to an additional physical symmetry.
If the flux of the fluid A is $f_A$, then the flux of the fluid B is $f_B=1-f_A$.
It can be represented with the operator
$
M:\R^{m^3}\rightarrow \R^{m^3}
$ be such that
\begin{equation*}
M(\xx)=(
1 - x_1, \dots, 1-x_{m^3})
\mbox{ for all }\mathbf x=(x_1, \dots, x_{m^3})\in \R^{m^3}.
\end{equation*}
When the roles of the fluids are switched (i.e. when $M$ is applied to the input vector), the output $ \alpha^\NN(M\xx)$ of the network should coincide with $1-\alpha^\NN(\xx)$.
Since it is not the case for a generic neural network such as (\ref{eq:didi}),
we define the  new neural network $\overline \alpha^\NN$
\begin{equation}\label{eq:net_switch}
\overline \alpha^\NN(\xx, \beta) = \frac 12 \alpha^\NN(\xx, \beta) + \frac12\left(1- \alpha^\NN(M\xx, \beta)\right)
\mbox{ for all } (\mathbf x, \beta)\in\mathbb R^{m^3} \times \mathbb R .
\end{equation}

\begin{lemma}\label{lem:lem2}
The neural network $\overline \alpha^\NN$  is additive in the sense
$$
\overline \alpha^\NN(\mathbf x, \beta)+\overline \alpha^\NN(M \mathbf x, \beta)=1 \mbox{ for all } (\mathbf x, \beta)\in \mathbb R^{m^3} \times \mathbb R .
$$
\end{lemma}

\begin{proof}
One has
\begin{equation*}
\begin{aligned}
\overline \alpha^\NN(M \xx, \beta) &= \frac 12 \alpha^\NN(M \xx, \beta) + \frac12\left(1- \alpha^\NN(M(M\xx, \beta))\right) \\
&= \frac 12 \alpha^\NN(M \xx, \beta) + \frac12\left(1- \alpha^\NN(\xx, \beta)\right) \\
&= -\frac 12 \alpha^\NN(\xx, \beta) + \frac12\left(1 + \alpha^\NN(M\xx, \beta)\right) \\
&= 1 - \left[\frac 12 \alpha^\NN(\xx, \beta) + \frac12\left(1- \alpha^\NN(M\xx, \beta)\right) \right] = 1 - \overline \alpha^\NN( \xx, \beta),
\end{aligned}
\end{equation*}
because $M\circ M(\mathbf x)=M(M\xx) = \xx$.
\end{proof}

To obtain a neural network satisfying both types of constraints, it is sufficient to combine the two operations by linear combination.
It yields the new neural network
\begin{equation}\label{eq:net_switch}
\widetilde \alpha^\NN(\xx,\beta) = \frac 12\widehat \alpha^\NN(\xx,\beta) + \frac12\left(1- \widehat \alpha^\NN(M\xx,\beta)\right) .
\end{equation}
By construction, the neural network $\widetilde \alpha^\NN$ also satisfies the symmetry constraints of
Lemma
\ref{lem:lem1}
and Lemma \ref{lem:lem2}.

\subsection{Non-negativity constraints}\label{sec:non-neg}

Since we are interested in using the neural network's predictions in \eqref{eq:fv_split} to update the volume fractions for the subsequent time step, we need to enforce the preservation of natural bounds.
In our case, such bounds come from the fact that a volume fraction is necessarily between 0 and 1 since $\alpha$
is an indicator function
$$
0\leq \alpha_C = \dfrac{1}{|C|}\int_C \alpha\, dx\, dy\, dz \leq 1
$$
which shows up at three different places,
where the first two items pose no problem. The real issue is the third one.

\begin{itemize}
\item For all $(\mathbf x, \beta)\in \mathcal D_{\rm inp}$, one has $\mathbf x\in [0,1]^{m^3}$. Note that a CFL number is also between 0 and 1, so that  $\mathcal D_{\rm inp}\subset [0,1]^{m^3+1}$. This constraint is satisfied by construction of the dataset.

\item Similarly, the exact flux is in bounds, that is $\alpha(\mathbf x, \beta)\in [0,1]$ for all $(\mathbf x, \beta)$. This constraint is also satisfied
by construction of the dataset. So one has by construction  $\mathcal D_{\rm out}\subset [0,1]$ and finally  $\mathcal D\subset [0,1]^{m^3+2}$.

\item More technical is the satisfaction of the maximum principle attached to the transport of the indicator function with the transport equation
(\ref{eq:prob}). The practical question is the following. Assume the initial discrete solution of the discrete scheme
(\ref{eq:fv_split}) is in bounds, that is
$\alpha_{i,j,k}\in [0,1]$ for all $(i,j,k)$. Then one must make sure that the final discrete solution is in bounds as well, that is
$\overline \alpha_{i,j,k}\in [0,1]$ for all $(i,j,k)$.
\end{itemize}


Our solution, inspired by
\cite{despres2001contact,despres2020machine}, is
to analyze the last of these three items, starting from the consideration of the first step of the discrete scheme (\ref{eq:fv_split}).
Let us consider a cubic cell with edge length  $\Delta x$ and with volume fraction $\alpha_{i,j,k}\in[0,1]$. The volume of the cell is $\Delta x^3$, the volume occupied by the fluid A is $0\le V_A=\alpha_{i,j,k}\Delta x^3$.
The volume occupied by the complementary fluid B is $0\le V_B=(1-\alpha_{i,j,k})\Delta x^3$.

The first step of the splitting scheme (\ref{eq:fv_split})
is
\begin{equation} \label{eq:cd1}
\alpha_{i,j,k}^{x} = \alpha_{i,j,k}  -  \beta_x
\left(\alpha_{{i+ \frac12,j,k}} - \alpha_{{i- \frac12 ,j,k}} \right), \qquad \beta_x=u_x \frac{    \Delta t }{ \Delta x}\geq 0.
\end{equation}
For convenience, we use a complementary variable  $\mu$ such that $\alpha+\mu=1$ everywhere. That is
$\alpha_{i,j,k} +\mu_{i,j,k} =1$, $\alpha_{{i\pm \frac12,j,k}}+\mu_{{i\pm \frac12,j,k}}=1$ and
$\alpha_{i,j,k}^{x}+\mu_{i,j,k}^{x}=1$.
Since we use a flux which satisfies the additivity condition of Lemma \ref{lem:lem2},
one can  write
\begin{equation} \label{eq:cd2}
\mu_{i,j,k}^{x} = \mu_{i,j,k}  -  \beta_x
\left(\mu_{{i+ \frac12,j,k}} - \mu_{{i- \frac12 ,j,k}} \right), \qquad \beta_x=u_x \frac{    \Delta t }{ \Delta x}\geq 0.
\end{equation}
Therefore the  double inequality $0\leq \alpha_{i,j,k}^{x}  \leq 1$ is  equivalent to
$0\leq \alpha_{i,j,k}^{x} $ and $0\leq \mu_{i,j,k}^{x}$.

\begin{lemma}
Assume the initial data is in bounds $0\leq  \alpha_{i,j,k}\leq 1$ for all $(i,j,k)$.
Assume the fluxes satisfy
$
0\leq \alpha_{i,j,k}^{x}  \leq 1$ for all $(i,j,k)$.
Assume the numerical fluxes satisfy
\begin{equation} \label{eq:nm4}
0\leq  \alpha_{i+\frac12,j,k}   \leq   \frac{1}{\beta_x} \alpha_{i,j,k}^{x}
\mbox{ and }
0\leq 1 -  \alpha_{i+\frac12,j,k}   \leq   \frac{1}{\beta_x} (1-\alpha_{i,j,k}^{x} ) \mbox{ for all }(i,j,k).
\end{equation}
Then the final data is in bounds $
0\leq \overline {\alpha_{i,j,k}}  \leq 1$ for all $(i,j,k)$.
\end{lemma}

The second set of inequalities can be  rewritten for the complementary variable
$0\leq  \mu_{i+\frac12,j,k}   \leq   \frac{1}{\beta_x} \mu_{i,j,k}^{x} $.

\begin{proof}
Let us make the hypothesis that only the bounds from below are satisfied
$
0\leq  \alpha_{i+\frac12,j,k}$ and $0\leq  \mu_{i+\frac12,j,k} $.
An easy interpretation is that non-negative volumes of fluids are exchanged between cells.
Then we observe that (\ref{eq:cd1}-\ref{eq:cd2}) implies
$$
\left\{
\begin{array}{lll}
\alpha_{i,j,k}^{x} \geq  \alpha_{i,j,k}  -  \beta_x
\alpha_{{i+ \frac12,j,k}}, \\
\mu_{i,j,k}^{x} \geq  \mu_{i,j,k}  -  \beta_x
\mu_{{i+ \frac12,j,k}} .
\end{array}
\right.
$$
One observes that $ \beta_x
\alpha_{{i+ \frac12,j,k}}$ (resp. $\beta_x
\mu_{{i+ \frac12,j,k}}$) is some non-negative quantity subtracted from the budget in the cell.
Of course, it is highly recommended to have non-negative cell fractions at the end of the process.
Then it is natural to enforce
$ \alpha_{i,j,k}  -  \beta_x
\alpha_{{i+ \frac12,j,k}} \geq 0$ and $\mu_{i,j,k}  -  \beta_x
\mu_{{i+ \frac12,j,k}} \geq 0$ which are nothing than the upper bounds in the hypothesis (\ref{eq:nm4}) of the Lemma.
One obtains $\alpha_{i,j,k}^{x}\geq 0$ and $\mu_{i,j,k}^{x} \geq 0$ which are equivalent to the claim.
\end{proof}

\begin{remark}
Using the method proposed in \cite{despres2001contact},
all four inequalities in (\ref{eq:nm4})
can be rewritten as
$$
m_{i,j,k}^x \leq  \alpha_{{i+ \frac12,j,k}} \leq
M_{i,j,k}^x
$$
where
$$
m_{i,j,k}^x= \max \left(0,1-   \frac{1}{\beta_x} (1-\alpha_{i,j,k}^{x} ) \right)
\mbox{ and }
M_{i,j,k}^x =
\min \left( 1,   \frac{1}{\beta_x} \alpha_{i,j,k}^{x} \right).
$$
These inequalities are
compatible
under CFL $\beta_x \leq 1$, because  it is evident to check that
the choice  $ \alpha_{{i+ \frac12,j,k}} = \alpha_{{i,j,k}} $ satisfies all inequalities.
\end{remark}

To enforce them, we adopt a simple projection strategy to ensure that the predicted flux $\alpha^\NN(\xx)$ is always included in the admissibility interval.
We predict the flux with the neural network and, if its output is outside the interval, we just select the extremum closer to the prediction.
The formal formula that incorporates all issues previously discussed is in the next definition.

\begin{definition}
The  VOFML flux in 3D is given by the formula
\begin{equation} \label{eq:fflux}
\alpha_{{i+ \frac12,j,k}}^{x, {\rm final} }= P_{[ m_{i,j,k}^x , M_{i,j,k}^x ]}
\widetilde \alpha^\NN(\xx,\beta_x) , \qquad \mathbf x= \left(\alpha_{i', j',k'} \right)_{(i',j',k')\in S_{i,j,k}^m}.
\end{equation}
The same procedure is repeated in  the directional splitting procedure for the directions $y$ and $z$.
If some components of the velocity are negative,  the stencil $ S_{i,j,k}^m $ is upwinded accordingly, and the Courant number is kept non-negative.
\end{definition}

\section{Numerical results}\label{sec:numerical_results}
In this section, we show and discuss numerical experiments empirically proving the better performance of the VOF-ML scheme.
Note that the implementation of the VOF-ML is done in the context of hybridization, see Appendix \ref{sec:hybrid-vof}. Hybridization simply means that
the VOF-ML scheme is activated close to the interface and that another standard and much less costly scheme is activated away from the interface. It results in some reduction of the implementation cost.

For a fair comparison, we only compare the VOF-ML scheme with other schemes exploiting the directional splitting and that do not explicitly reconstruct the interface. In particular, we test the commonly used upwind scheme (UW) and the limited downwind scheme (LD). The former, despite its popularity, is known to be excessively diffusive for sharp interface advection problems, whereas the latter is very accurate and can even exactly reconstruct interfaces that are aligned with the grid. 

The correctness of our implementation has been verified on a simple test (not reported here) with constant velocity $u=[1,1,1]^T$ and where the region A is the cube $[0.2, 0.8]^3$ inside a cubic domain $\Omega=(0,1)^3$.  In particular, we have checked that our implementation satisfies two properties of the upwind scheme: firstly the upwind scheme is exact for a Courant number equal to 1, and secondly the theoretical asymptotic convergence rate is 0.5 in the $L^1$ norm for a Courant number less than 1 \cite{delarue2011probabilistic}. This rate is the asymptotic rate of convergence. However, in our numerical tests we observe a pre-asymptotic convergence rate $<0.5$ which is interpreted as the effect of the excessive numerical diffusion due to the specifc velocity and initial conditions. 


In \cite{despres2020machine, ancellin2023extension}, the authors observed that the VOF-ML performances were significantly better than the upwind ones and slightly better than the limited downwind ones for two-dimensional problems. We now show that, for three-dimensional problems, where more complex interfaces may arise, the advantage of the VOF-ML scheme is much greater.

In all the numerical tests, we use the same neural network (without any retraining or fine-tuning). Using the notation of Section \ref{sec:network}, the chosen architecture is characterized by the following hyperparameters:
\[
\rho(x) = ReLU(x), \quad L = 5, \quad
a_0 = 3^3+1, \quad a_L=1, \quad a_i=50, \,\forall i=1,2,3,4.
\]
Such a network contains 9151 trainable coefficients, which are optimized by minimizing \eqref{eq:loss} through 5000 epochs performed with the ADAM optimizer \cite{kingma2014adam}, followed by 5000 steps of the BFGS optimizer \cite{wright1999numerical}. To be coherent with the fact that the distributions used to create the dataset have different parametrizations, we generate the training dataset using the set of synthetic distributions in Table \ref{tab:dataset}, where a random Courant number 
$$\beta \in [0, 0.6]$$
 is associated with each distribution.

\begin{table}[!ht]
\centering
\begin{tabular}{|c|c|c|}
\hline
Number of samples &Generating geometrical objects &Number of associated parameters\\ \hline\hline
3000 & 1 plane & 3 \\\hline
6000 & 2 planes & 6 \\\hline
9000 & 3 planes & 9 \\\hline
6000 & 1 ellipsoid & 6\\\hline
\end{tabular}
\caption{Distributions used to generate the training dataset $\cal D$.}
\label{tab:dataset}
\end{table}

As discussed in Section \ref{sec:dataset}, to each of these distributions we associate the vector $\mathbf x$ of the corresponding volume fractions. Moreover, to enrich the dataset, for each of its element, we add all the vectors of the form $(\sigma \mathbf x,\beta,\gamma_\sigma)$, for all $\sigma\in{\cal S}_{\neq}$ as discussed in Section \ref{sec:symms}. The dimension of the dataset is thus $(3000+6000+9000+6000)\cdot6=144000$.

As usually done, the entire dataset is split into three smaller dataset: the training dataset ($80\%$ of the data) used to optimize the weights, the validation dataset ($10\%$ of the data) used to avoid overfitting, and the test dataset ($10\%$ of the data) used to test the model.  In particular, we compare the flux reconstruction capability of the three considered methods on the test dataset and show the results in Table \ref{tab:test_compar}. For each method, we compute the Mean Squared Error (MSE) and the Mean Absolute Error (MAE) defined as:
\[
{\rm{MSE}} = \frac{1}{\#{\cal D_{\rm test} }} \sum_{(\boldsymbol{x},\beta)\in{\cal D_{\rm test} }} \left( \alpha^*(\xx,\beta) - \alpha(\xx,\beta)\right)^2, \quad
{\rm{MAE}} = \frac{1}{\#{\cal D_{\rm test} }} \sum_{(\boldsymbol{x},\beta)\in{\cal D_{\rm test} }} \left\vert \alpha^*(\xx,\beta) - \alpha(\xx,\beta)\right\vert,
\]
where $\cal D_{\rm test}$ is the test dataset,  $\#{\cal D_{\rm test}}$ is its size, and $ \alpha^*$ denotes the flux computed with one of the three methods. In Table \ref{tab:test_compar}, we highlight that the flux reconstruction is much more accurate, using both types of errors, when the VOFML method is used. In particular, when considering the MSE metric, VOFML is 43.7 times more accurate than UW and 13.2 times more accurate than LD, whereas when the MAE metric is involved, VOFML is 8.57 times more accurate than UW and 3.70 times more accurate than LD. Such an error is propagated at each time step during the simulation, it is thus crucial to reconstruct the flux in the most accurate way possible.

\begin{table}[!ht]
\centering
\begin{tabular}{|ccc||ccc|}
\hline
\multicolumn{3}{|c||}{MSE} & \multicolumn{3}{c|}{MAE} \\ \hline
\multicolumn{1}{|c|}{UW} & \multicolumn{1}{c|}{LD} & VOFML & \multicolumn{1}{c|}{UW} & \multicolumn{1}{c|}{LD} & VOFML \\ \hline\hline
\multicolumn{1}{|c|}{1.339e-02} & \multicolumn{1}{c|}{4.048e-03} & 3.071e-04  & \multicolumn{1}{c|}{3.995e-02} & \multicolumn{1}{c|}{1.724e-02} & 4.660e-03 \\ \hline
\end{tabular}
\caption{Flux reconstruction error on the test dataset with the different metrics and methods.}
\label{tab:test_compar}
\end{table}

\subsection{Test 1: constant velocity}
Let us consider the domain $\Omega=(-1,1)^3$ and the problem
\begin{equation}\label{eq:prob_const_u}
\partial_t \alpha + \nabla\cdot(\alpha u) = 0,
\end{equation}
with $u=[1,2,3]^T$, periodic boundary conditions on the entire boundary $\partial \Omega$ and initial condition $\alpha_0$. To mimic the Zalezak's initial condition in this three-dimensional domain, let us consider a sphere, from which we subtract a rectangular prism. Let us denote the two three-dimensional regions $R_1$ and $R_2$ as:
\[
R_1 = \{(x,y,z)\in\Omega : x^2 + y^2 + z^2 < 0.4^2\},
\]
\[
R_2 = \{(x,y,z)\in\Omega : |x|<0.2, |y|<0.2, z<0\}.
\]
Then, the region $R_\text{zal}$ is defined as $R_\text{zal} = R_1 \backslash R_2$. The initial condition $\alpha_0$ is the indicator function of the set $\widehat R_\text{zal}$, obtained from $R_\text{zal}$ through a rotation $R_0$ around the center of the domain to avoid mesh-related artifacts during the simulation. The rotation $R_0$ can be obtained combining a anti-clockwise rotations of $\pi/5$, $\pi/7$ and $\pi/9$ radians around the $x$, $y$ and $z$-axis, respectively.

The domain is discretized with a uniform Cartesian grid, with $N_h$ cells in each direction and a total of $N_h^3$ cells. The considered values of $N_h$ are 10, 14, 20, 27, 38, 54, 75,  and 105. The values are chosen to obtain almost equally spaced points in the plots of the error, which are in logarithmic scale.  The initial
volume fractions in each cell are computed with tensorial quadrature rules, higher order quadrature rules would be equally accurate because of the discontinuous nature of the initial condition.

We are interested in simulating the dynamic up to time $T=2$. This way, the exact distribution $\alpha(\cdot, T)$ at time $t=T$ coincides with the initial distribution $\alpha(\cdot, 0)$ at time $t=0$. Indeed, due to the periodic boundary conditions, the initial distribution $\alpha_0$ is advected through the entire domain exactly one, two, and three times in the $x$, $y$, and $z$ direction, respectively. This is useful to evaluate the accuracy of the method by computing the error with respect to the exact solution.

We solve problem \eqref{eq:prob_const_u} with the three proposed methods (upwind, limited downwind, and VOF-ML) on different meshes to understand the convergence with respect to the mesh size. In order to fix the CFL condition, in this case equal to 0.1, 0.2, and 0.3 in the three directions, we fix $\Delta t = 0.1 \Delta x$. At the end of each simulation, we measure the accuracy by computing the three relative $L^1$ errors
\[
E^\text{UW} = \frac{ \Vert \alpha^\text{UW}(T) - \alpha_0\Vert_1}{\Vert \alpha_0\Vert_1}, \hspace{1cm}
E^\text{LD} = \frac{ \Vert \alpha^\text{LD}(T) - \alpha_0\Vert_1}{\Vert \alpha_0\Vert_1}, \hspace{1cm}E^\text{VOF-ML} = \frac{ \Vert \alpha^\text{VOF-ML}(T) - \alpha_0\Vert_1}{\Vert \alpha_0\Vert_1}.
\]
Here $\alpha^\text{UW}(T)$ and $\alpha^\text{LD}(T)$ are the solutions at time $t=T$, obtained with the upwind scheme and the limited downwind scheme, respectively. Instead, $\alpha^\text{VOF-ML}(T)$ is the solution, at time $t=T$, obtained with the hybrid method described in Appendix \ref{sec:hybrid-vof}. The fluxes close to the interface are computed with the neural network, whereas the other ones are computed with the limited downwind scheme.

The results are shown in Fig. \ref{fig:constant_vel_error}. It is evident that the upwind scheme is not able to correctly simulate the dynamic. The excessive diffusion leads to a very large error, which decreases slowly when the mesh is refined. On the other hand, the limited downwind scheme and the VOF-ML perform very similarly when the mesh is coarse enough ($N_h\le 20$ for this specific test), but for finer meshes the VOF-ML method converges faster with respect to the number of cells. The empirical convergence rates with respect to the number of cells, highlighted inside the brackets in the legend of the figure, are 0.16, 0.58, and 1.01 for the upwind, limited downwind, and VOF-ML method, respectively.

\begin{figure}[!ht]
\centering
\includegraphics[width=0.7\columnwidth,keepaspectratio,clip]{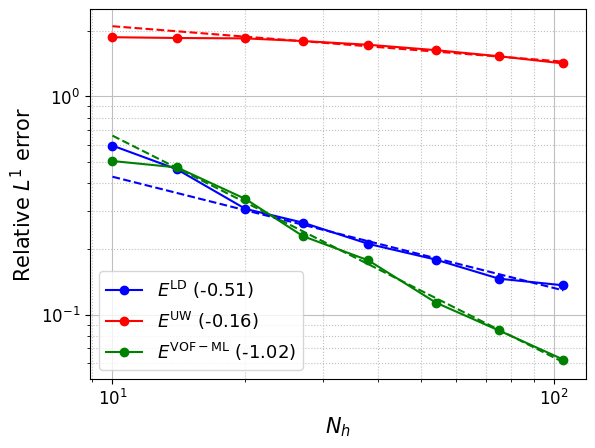}
\caption{Errors for Test 1. The number in the brackets is the convergence rate and the slope of the corresponding regression dashed line.}
\label{fig:constant_vel_error}
\end{figure}

In Figure \ref{fig:mix_cells_constant}, instead, we analyze the evolution of the ratio $R_\text{mix}$ between the number of mixed cells and the total number of cells in the mesh for three different refinements. Note that, in this first test problem, the initial condition is simply translated. Therefore, the size of the exact interface between the two fluids remains the same at any time instant $t\in[0,T]$, and $R_\text{mix}$ should remain approximately constant during the entire simulation. In the subfigures, we report results for the coarsest mesh ($N_h=10$), an intermediate mesh ($N_h=27$), and the finest mesh ($N_h=105$). We can observe that the noise on the curves decreases when the mesh is refined.

Unsurprisingly, $R_\text{mix}$ grows when the upwind scheme is adopted. This happens because the upwind scheme is extremely diffusive and is not able to preserve the sharp interfaces present in the initial condition. On the other hand, $R_\text{mix}$ remains approximately constant with both the limited downwind scheme and the VOF-ML scheme, pretty close to the value associated with the initial condition. The fact that the lowest values of $R_\text{mix}$ are obtained with the limited downwind scheme, even if its overall performance is worse than the VOF-ML one, can be explained by the fact that the scheme is developed to exactly advect one-dimensional interfaces. The diffusion in each update of the directional splitting is thus minimized when the limited downwind scheme is adopted. On the other hand, no information about the diffusion is used to train the neural network, and the constant behaviour of the associated $R_\text{mix}$ is only a consequence of its high accuracy.

\begin{figure}[!ht]
\centering
\begin{subfigure}{0.32\textwidth}
\includegraphics[width=\textwidth]{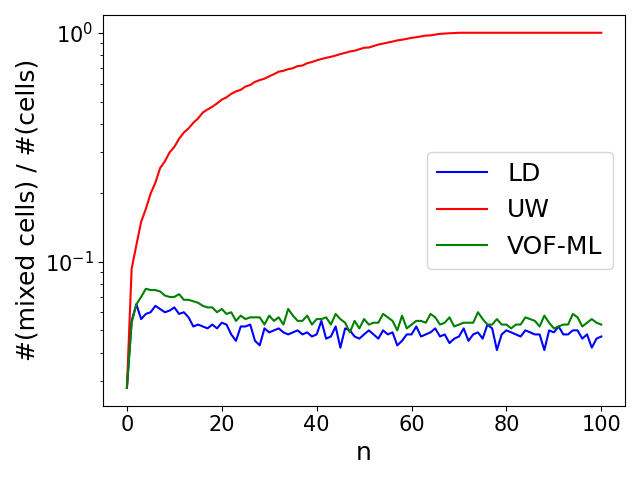}
\caption{$N_h=10$.}
\label{}
\end{subfigure}
\begin{subfigure}{0.32\textwidth}
\includegraphics[width=\textwidth]{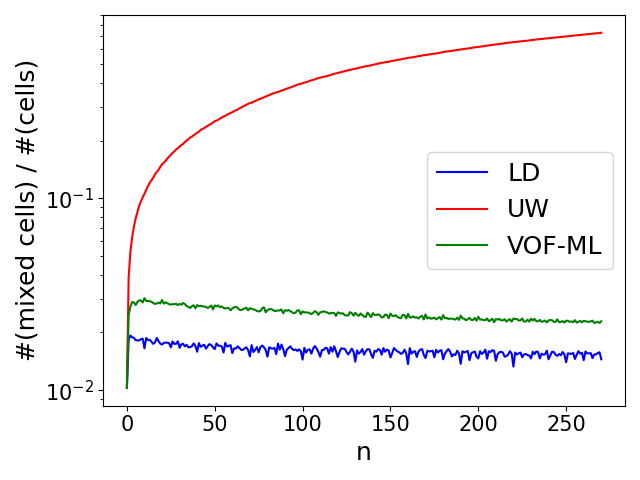}
\caption{$N_h=27$.}
\label{}
\end{subfigure}
\begin{subfigure}{0.32\textwidth}
\includegraphics[width=\textwidth]{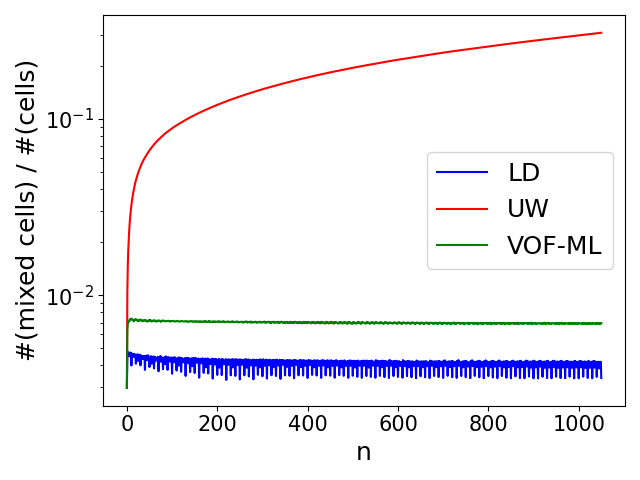}
\caption{$N_h=105$.}
\label{}
\end{subfigure}
\caption{Test 1. Ratio of mixed cells over the total number of cells. Evolution (in logarithmic scale) during the simulation for a coarse, intermediate, and fine mesh.}
\label{fig:mix_cells_constant}
\end{figure}

We also show, for the same meshes, the values of the ratio $R_\text{mix}^{T/0}$ between $R_\text{mix}$ at the end of the simulation and the initial condition. A value larger than 1 means that the interface between the two fluids expanded because more mixed cells are present, whereas a value smaller than 1 means that the interface shrank, for the opposite reason. In Table \ref{tab:rmix_ratio_constant}, we can observe that $R_\text{mix}^{T/0}$ decreases only with the limited downwind scheme when the mesh is refined. However, with both the limited downwind scheme and the VOF-ML scheme, $R_\text{mix}^{T/0}$ is of order 1, whereas it increases dramatically with the upwind scheme, highlighting once more the excessive diffusion of the method.

\begin{table}[!ht]
\centering
\begin{tabular}{|c||c|c|c|}
\hline
$N_h$ & \hphantom{AA}UW\hphantom{AA} & \hphantom{AA}LD\hphantom{AA} & VOF-ML \\ \hline\hline
10 &  35.71 &  1.68 & 1.89 \\ \hline
27 & 71.12 & 1.41 & 2.23 \\ \hline
105 & 103.3 & 1.14 & 2.33 \\ \hline
\end{tabular}
\caption{Test 1. Ratios $R_\text{mix}^{T/0}$ for different meshes.}
\label{tab:rmix_ratio_constant}
\end{table}

\subsection{Test 2: velocity fields $u=[u_1, u_2, u_3]$ with $\frac{\partial u_1}{\partial x}=\frac{\partial u_2}{\partial y} =\frac{\partial u_3}{\partial z} =0$}
Let us consider a slightly more general test case. We consider equation \eqref{eq:prob_const_u} on the domain $\Omega=(0,1)^3$, before the time limit $T=1$, with periodic boundary conditions and with a velocity field $u=[u_1, u_2, u_3]$ such that:
\begin{equation}\label{eq:vel_directions}
\begin{aligned}
u_1 &= 25 \sin(2\pi y)^2\sin(2\pi z)  y(y-1)z(z-1)  \cos(\pi t/T), \\
u_2 &= 25\sin(2\pi z)^2 \sin(2\pi x)  z(z-1) x(x-1) \cos(\pi t/T), \\
u_3 &= 25 \sin(2\pi x)^2\sin(2\pi y)   x(x-1)y(y-1) \cos(\pi t/T).
\end{aligned}
\end{equation}
Note that $\frac{\partial u_1}{\partial x}=\frac{\partial u_2}{\partial y} =\frac{\partial u_3}{\partial z} =0$. It is thus evident that $u$ is divergence-free, but the stricter property on the three derivatives ensures that, during the directional splitting, the intermediate updates are perfectly conservative (up to machine precision). This is sufficient to ensure global conservation, even if a directional splitting approach is adopted.

Inspired by an initial distribution considered in \cite{ancellin2023extension}, let us also consider a more complex initial distribution. Let us define the following regions:
\[
R_\text{sphere} = \{(x,y,z)\in\Omega : (x-0.5)^2 + (y-0.5)^2 + (z-0.5)^2 < 0.2^2\},
\]
\[
R_{3,x} =  \{(x,y,z)\in\Omega : |x-0.5|<0.3, |y-0.5|<0.075, |z-0.5|<0.075\},
\]
\[
R_{3,y} =  \{(x,y,z)\in\Omega : |x-0.5|<0.075, |y-0.5|<0.3, |z-0.5|<0.075\},
\]
\[
R_{3,z} =  \{(x,y,z)\in\Omega : |x-0.5|<0.075, |y-0.5|<0.075, |z-0.5|<0.3\},
\]
We define $R_\text{init}$ as $R_\text{init}=R_\text{sphere}\cup R_{3,x}\cup R_{3,y}\cup R_{3,z}$. As in the previous case, the initial condition $\alpha_0$ is the indicator function of the set $\widehat R_\text{init}$ obtained from $R_\text{init}$ through the rotation $R_0$ defined before.

As in the previous case, we fix the time step size $\Delta t$ as $\Delta t = 0.1 \Delta x$ and perform three simulations for each mesh, one with the upwind scheme, one with the limited downwind scheme, and one with the VOF-ML scheme. The results are reported in Figure \ref{fig:directions_error}. Once more, the upwind scheme is excessively diffusive and not competitive with the other schemes. The limited downwind scheme and the VOF-ML scheme reach similar accuracies on coarse meshes, but the latter is characterized by a faster convergence rate with respect to $N_h$ and rapidly becomes more accurate than the former when the mesh is refined.

\begin{figure}[!ht]
\centering
\includegraphics[width=0.7\columnwidth,keepaspectratio,clip]{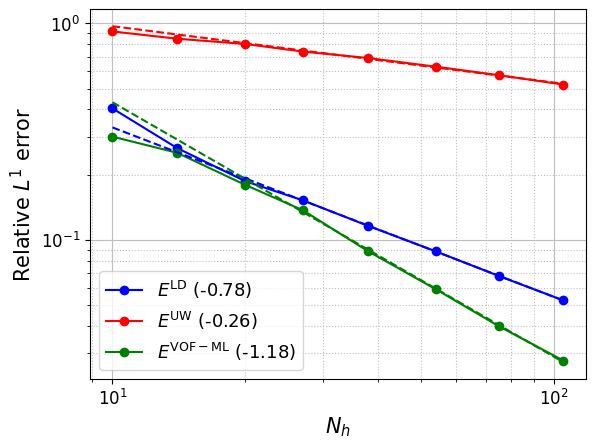}
\caption{Errors for Test 2. The number in the brackets is the convergence rate and the slope of the corresponding regression dashed line.}
\label{fig:directions_error}
\end{figure}

In Figure \ref{fig:mix_cells_directions}, we show the evolution of $R_\text{mix}$ on this second test. Once more, it is evident that the upwind scheme is excessively diffusive, whereas the limited downwind scheme and the VOF-ML scheme are able to capture the sharp interfaces. Note that, in this case, the initial distribution is deformed by the velocity field \eqref{eq:vel_directions} and therefore $R_\text{mix}$ should not remain constant. However, because of the factor $\cos(\pi t/T)$ in \eqref{eq:vel_directions}, the velocity field is symmetrical, in time, around the instant $t=T/2$. This implies that the initial distribution, deformed up to time $t=T/2$, is deformed back to the initial condition exactly in the same way, implying, as a consequence, a symmetrical behaviour in the expected evolution of $R_\text{mix}$. This symmetry is more evident in Figure
\subref{fig:mix-cells-directions-fine}
because the fine mesh minimizes the additional noise that may ruin it.

\begin{figure}[!ht]
\centering
\begin{subfigure}{0.32\textwidth}
\includegraphics[width=\textwidth]{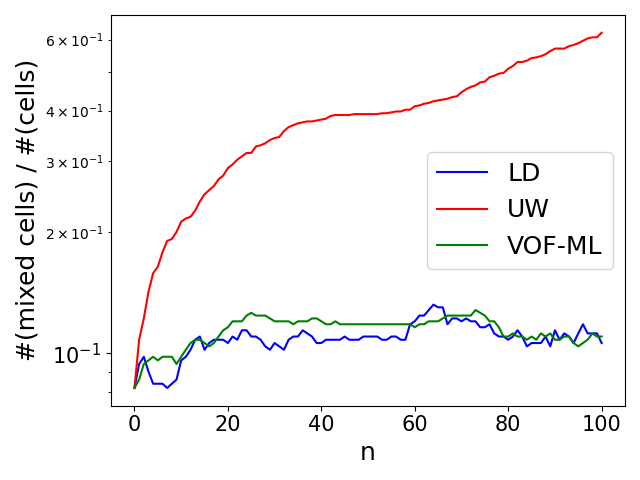}
\caption{$N_h=10$.}
\label{}
\end{subfigure}
\begin{subfigure}{0.32\textwidth}
\includegraphics[width=\textwidth]{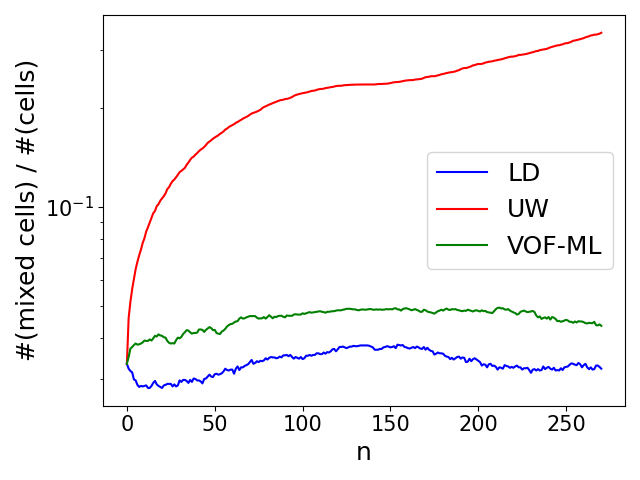}
\caption{$N_h=27$.}
\label{}
\end{subfigure}
\begin{subfigure}{0.32\textwidth}
\includegraphics[width=\textwidth]{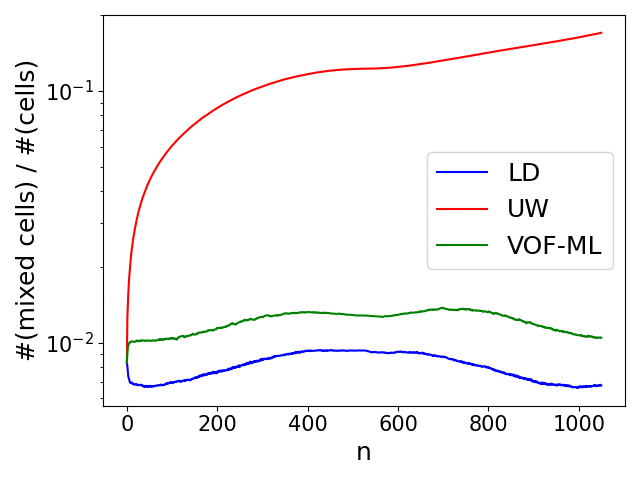}
\caption{$N_h=105$.}
\label{fig:mix-cells-directions-fine}
\end{subfigure}
\caption{Test 2. Ratio of mixed cells over the total number of cells. Evolution (in logarithmic scale) during the simulation for a coarse, intermediate, and fine mesh.}
\label{fig:mix_cells_directions}
\end{figure}

The ratios $R_\text{mix}^{T/0}$ are shown in Table \ref{tab:rmix_ratio_directions}. Once more, only the upwind scheme's values grow excessively. It is interesting to remark that, in this case, $R_\text{mix}^{T/0}$ decreases below 1 for the upwind scheme, highlighting a slight antidiffusivity property of the method. In particular, the values of the VOF-ML scheme and of the limited downwind scheme for the last mesh are almost equidistant from the target value 1.

\begin{table}[!ht]
\centering
\begin{tabular}{|c||c|c|c|}
\hline
$N_h$ & \hphantom{AA}UW\hphantom{AA} & \hphantom{AA}LD\hphantom{AA} & VOF-ML \\ \hline\hline
10 &  7.61 & 1.29 & 1.34 \\ \hline
27 & 10.14 & 0.97 & 1.30 \\ \hline
105 & 20.45 & 0.81 & 1.26 \\ \hline
\end{tabular}
\caption{Test 2. Ratios $R_\text{mix}^{T/0}$ for different meshes.}
\label{tab:rmix_ratio_directions}
\end{table}

\subsection{Test 3: general divergence-free velocity fields}
Lastly, inspired by a test in \cite{fakhreddine2024directionally}, we consider a velocity field $u=[u_1, u_2, u_3]$ where
\begin{equation}\label{eq:vel_variable}
\begin{aligned}
u_1 &= 2\sin(\pi x)^2 \sin(2\pi y) \sin(2\pi z)  \cos(\pi t/T), \\
u_2 &= -\sin(\pi y)^2 \sin(2\pi x) \sin(2\pi z) \cos(\pi t/T), \\
u_3 &= -\sin(\pi z)^2 \sin(2\pi x) \sin(2\pi y) \cos(\pi t/T).
\end{aligned}
\end{equation}
Such velocity field is divergence-free, but the derivatives $\frac{\partial u_1}{\partial x}$, $\frac{\partial u_2}{\partial y}$ and $\frac{\partial u_3}{\partial z}$ are different than 0. This implies that the velocity used in the update in a single direction is not divergence-free, and therefore the mass conservation may be violated.

Let us denote by $\alpha_A$ and $\alpha_B$ the volume fractions of fluid A and B, respectively. Our goal is to ensure that $0\le\alpha_A\le1$, $0\le\alpha_B\le1$, $\alpha_A+\alpha_B=1$ in each cell and at each time step. As discussed in Section \ref{sec:non-neg}, a simple flux projection is sufficient to avoid negative values of $\alpha_A$ and $\alpha_B$, but a more elaborate strategy is needed to ensure the other constraints.

Let us assume that, at a given instant, all volume fractions $\alpha^n$ are in the interval $[0,1]$, and let us define $\alpha_A^n=\alpha^n$ and $\alpha_B^n=1-\alpha^n$. Because of the interchangeability of the fluids, both $\alpha_A$ and $\alpha_B$ satisfy the equation:
\[
\partial_n \alpha_\dag + \nabla\cdot(\alpha_\dag u) = 0, \hspace{1cm} \dag\in\{A,B\},
\]
that can be discretized as discussed in Section \ref{sec:nn_in_FV} via directional splitting. After a single directional update for both $\alpha_A$ and $\alpha_B$, we obtain the intermediate volume fractions $\widetilde \alpha_A^{n+1}$ and $\widetilde \alpha_B^{n+1}$ which may be greater than 1 and whose sum may be different than 1. We thus compute the normalized volume fractions $\alpha_A^{n+1}$ and $\alpha_B^{n+1}$ as
\[
\alpha_A^{n+1} = \frac{\widetilde \alpha_A^{n+1}}{\widetilde \alpha_A^{n+1} + \widetilde \alpha_B^{n+1}}, \hspace{1cm}
\alpha_B^{n+1} = \frac{\widetilde \alpha_B^{n+1}}{\widetilde \alpha_A^{n+1} + \widetilde \alpha_B^{n+1}}.
\]
Note that $\alpha_A^{n+1}$ and $\alpha_B^{n+1}$ satisfy all the required constraints. It is thus possible to define $\alpha^{n+1}$ as $\alpha^{n+1}=\alpha_A^{n+1}$ to obtain an update that does not violate anymore the mass conservation constraints. Therefore, the described procedure is able to preserve the mass conservation of the initial volume fractions for a single directional update. Since the initial condition of our problem at time $t=0$ does not violate the mass conservation constraints, the described procedure can be iteratively applied to advect the initial condition while preserving the mass of both fluids.

For efficiency reasons, we highlight that, after having computed a flux $\alpha_A^*$ to update $\alpha_A^n$, the flux $\alpha_B^*=1-\alpha_A^*$ is already available and does not require any further (possibly expensive) computations. For this reason, the computational cost of the procedure is almost the same as the one of the simple update used in the previous sections.

We can finally consider the problem
\begin{equation*}
\partial_t \alpha + \nabla\cdot(\alpha u) = 0,
\end{equation*}
with the velocity defined in \eqref{eq:vel_variable}, on the domain $\Omega=(0,1)^3$, up to time $T=2$, and with the initial condition as the indicator function of a sphere centered in $[0.35, 0.35, 0.35]$ and radius 0.15.

The convergence rates are shown in Figure \ref{fig:general_error}. We can observe that, due to the more complex nature of the problem, a more evident pre-asymptotic phase is present for meshes with less than $N_h=20$ cells in each direction. However, the error decays are similar to the ones shown in the previous cases, with the VOF-ML that significantly outperforms the limited downwind scheme and the upwind scheme that is excessively diffusive and converges very slowly.

\begin{figure}[!ht]
\centering
\includegraphics[width=0.7\columnwidth,keepaspectratio,clip]{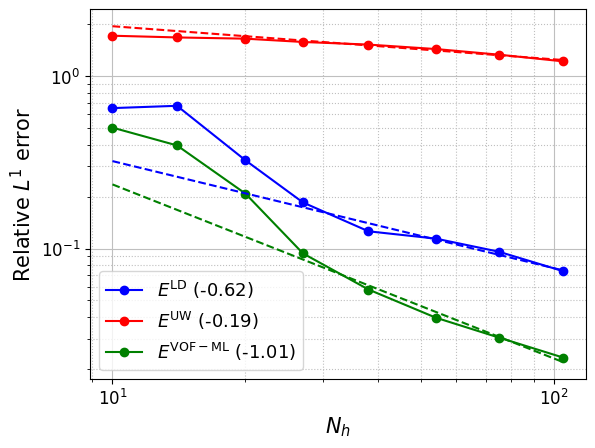}
\caption{Errors for Test 3. The number in the brackets is the convergence rate and the slope of the corresponding regression dashed line.}
\label{fig:general_error}
\end{figure}

Also in this case, we can analyze the behaviour of $R_\text{mix}$, plotted in Figure \ref{fig:mix_cells_general}. As in the previous test problem, the upwind scheme is not able to accurately preserve the initial discontinuous interface, and the evolution of $R_\text{mix}$ is symmetrical when the limited downwind or the VOF-ML schemes are adopted. Because of the similarity between Figure \ref{fig:mix_cells_directions} and Figure \ref{fig:mix_cells_general}, we can infer that the additional normalization introduced in this last test case does not interfere with the previously observed properties of the methods.

\begin{figure}[!ht]
\centering
\begin{subfigure}{0.32\textwidth}
\includegraphics[width=\textwidth]{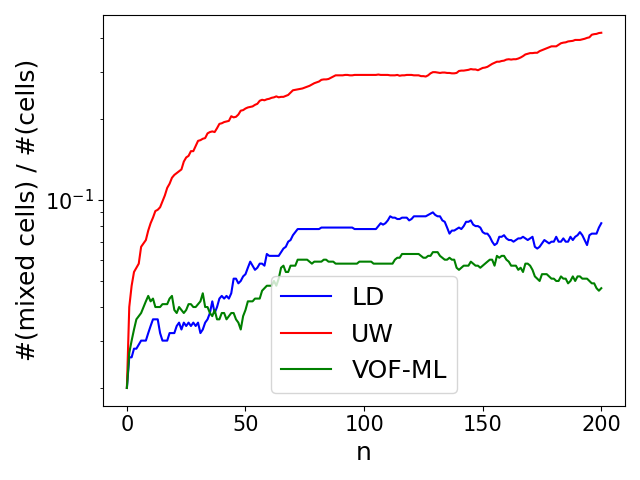}
\caption{$N_h=10$.}
\label{}
\end{subfigure}
\begin{subfigure}{0.32\textwidth}
\includegraphics[width=\textwidth]{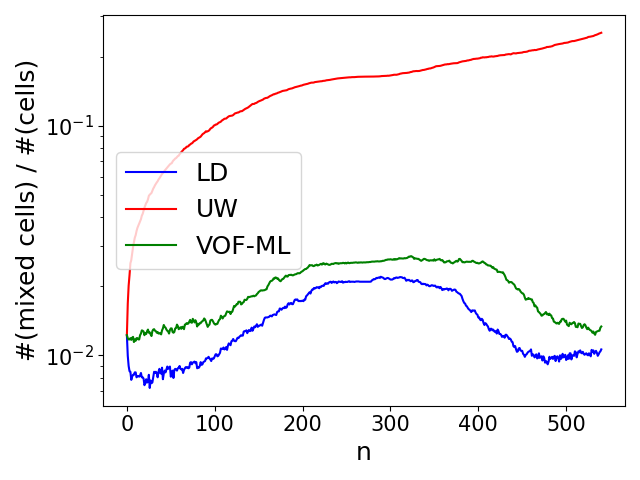}
\caption{$N_h=27$.}
\label{}
\end{subfigure}
\begin{subfigure}{0.32\textwidth}
\includegraphics[width=\textwidth]{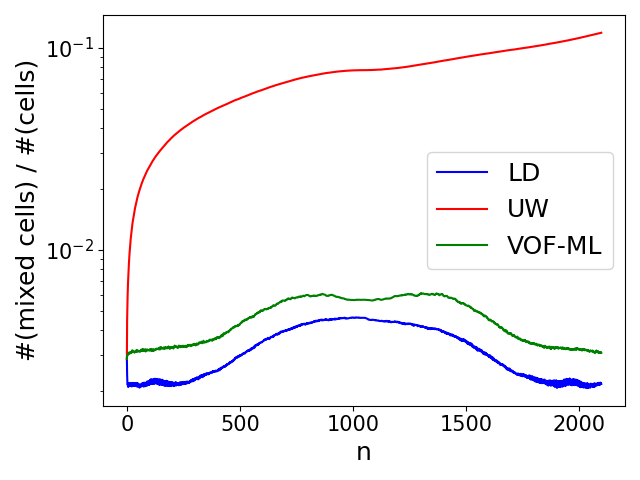}
\caption{$N_h=105$.}
\label{}
\end{subfigure}
\caption{Test 3. Ratio of mixed cells over the total number of cells. Evolution (in logarithmic scale) during the simulation for a coarse, intermediate, and fine mesh.}
\label{fig:mix_cells_general}
\end{figure}

In Table \ref{tab:rmix_ratio_general} we show the values of $R_\text{mix}^{T/0}$. As in the previous cases, only the limited downwind and the VOF-ML display values close to 1, with the VOF-ML scheme that is particularly close to such a target value.

\begin{table}[!ht]
\centering
\begin{tabular}{|c||c|c|c|}
\hline
$N_h$ & \hphantom{AA}UW\hphantom{AA} & \hphantom{AA}LD\hphantom{AA} & VOF-ML \\ \hline\hline
10 &  20.95 & 4.10 & 2.35  \\ \hline
27 &  20.75 & 0.87 & 1.09 \\ \hline
105 & 41.13 & 0.75 & 1.07  \\ \hline
\end{tabular}
\caption{Test 3. Ratios $R_\text{mix}^{T/0}$ for different meshes.}
\label{tab:rmix_ratio_general}
\end{table}

\section{Conclusion}\label{sec:conclusion}
In this paper, we have presented a machine learning based three-dimensional Volume Of Fluid (VOF) scheme on Cartesian grids, where a fully-connected feed-forward neural network is used to approximate the unknown function mapping the volume fractions and the Courant number to the target flux used in the Finite Volume update. Once the network is trained, it is therefore possible to directly evaluate the flux, without the need to reconstruct the interface in advance, which is the current bottleneck of most VOF schemes. Moreover, the absence of the interface reconstruction step allows a greater flexibility because it allows the method to accurately handle less regular interfaces, commonly present in the presence of more than two materials, where classical linear interfaces are not enough to provide good approximations.

The implementation of the method is also easier than most three-dimensional VOF schemes, because the flux is directly computed by evaluating a trained neural network, which can be efficiently done thanks to any of the several existing machine-learning libraries. In order to train the network, we propose a standard fitting procedure on a synthetic dataset, which can be generated by using only geometrical information. This allows the user to train a single network and to use it to solve different equations, on different domains, and with different initial and boundary conditions without any retraining. Nevertheless, the training dataset can be adapted, for example, by considering more or less non-regular interfaces, if some additional knowledge is available. A great focus is also devoted to the enforcement of physical constraints, which is achieved by suitably modifying the neural network.

The method has been validated on advection problems with accurate results for initial condition problems. Future perspectives include the analysis of the scheme on more physically relevant problems and the comparison with more complex or machine-learning based VOF schemes.
One of our next goals of research will be to couple the tridimensional  VOF-ML scheme to a more general hyperbolic scheme for solving compressible Euler equations with internal free boundaries in 3D.

\section*{Acknowledgements}
This work is funded by PEPR/IA  (https://www.pepr-ia.fr/).
Both authors thank St\'ephane Jaouen for fruitful discussions during the elaboration stage of this work.

\bibliographystyle{siam}
\bibliography{bibliography}

\appendix

\section{Parametrization of the training configurations}\label{sec:params}
 The considered training configurations are parametrized in the following way, in order for them to be described by parameters that vary in suitable hyper-cubes. This is important to efficiently sample random configurations in these high-dimensional parameter spaces. We always consider a normalized stencil, i.e. a stencil centered in the origin and comprising cubic cells of edge 1.

We start by introducing an important sampling strategy to uniformly sample points on the unit sphere, because it will be widely used in the following discussion.

\begin{remark}\label{rem:sphere_sampl}
A strategy to uniformly sample points $[a,b,c]\in\R^3$ on the unit sphere is to project a uniform distribution of points $[\alpha,\beta,\gamma]\in\R^3$ on a cylinder to the sphere. Such a simple method is based on Archimedes' theorem, see \cite{pitman2012archimedes} for more details. Given $\theta_1\in[0,2\pi]$ and $\theta_2\in[-1,1]$, the distribution on a cylinder of height 2 and radius 1 is obtained as
$
a = \cos(\theta_1)$, $
b = \sin(\theta_1)$ and $
c = \theta_2$.
It  is then mapped to the unit sphere as
$
\alpha = \sqrt{1-c^2} \cdot a$, $
\beta = \sqrt{1-c^2} \cdot b$ and
$\gamma = c$.
Only two parameters are needed to parameterize such a distribution.
\end{remark}

We now describe the parametrization of the selected distributions.

\begin{itemize}
\item The region A is a half-space. The plane separating the region A and the region B is parametrized as:

\[
p_1(x,y,z;\theta) = a_1(\theta)x + b_1(\theta)y + c_1(\theta)z - d_1(\theta),
\]
where $\theta=[\theta_1,\theta_2,\theta_3]\in \R^3$ is a vector of free parameters and the region A is defined as the set of points $(x,y,z)\in\R^3$ such that $p_1(x,y,z;\theta) <0$.

In order to obtain a uniform distribution of all the existing planes, we want $n_1(\theta) = [a_1(\theta),b_1(\theta),c_1(\theta)]$ to be uniformly distributed on the unit sphere (i.e. $\Vert n_1(\theta)\Vert_2=1$) and $d_1(\theta)$ to vary between in $[0,  \frac{\Vert n_1(\theta)\Vert_1}{2}]$. This ensures that the region A and the region B coexist in the central cell of the stencil.

The first two parameters $\theta_1$ and $\theta_2$ are involved in the sampling of all possible vectors $n_1(\theta)$ as in Remark \ref{rem:sphere_sampl}, whereas $d_1$ is computed as $d_1(\theta) = \theta_3 \frac{\Vert n_1(\theta)\Vert_1}{2}$, with $\theta_3\in[0,1]$. Therefore, the parameter $\theta$ must assume values in $[0,2\pi]\times[-1,1]\times[0,1]$.

Examples of distributions that can be generated with this approach are represented in Figure \ref{fig:1planes}.

\item The region A is the intersection between two half-spaces. The interfaces of the two half-spaces can be defined as:
\[
p_i(x,y,z;\theta) = a_i(\theta)x + b_i(\theta)y + c_i(\theta)z - d_i(\theta), \hspace{1cm} i=1,2,
\]
with $\theta\in\R^6$. Note that it is always possible to sample the two planes independently with the strategy described in the previous case, but this would lead to two main issues. It might happen that the region A, if non-empty, is outside the stencil, or otherwise it might coincide with a half-space, at least in the central cell. To avoid these issues and ensure that the two planes intersect inside the central cell, we have to modify the sampling strategy. The main idea is to generate two parametric planes intersecting on the $y$-axis and to apply parametric rotations and translations to move them such that the origin is mapped into a random point inside the central cell.

Given the first parameter $\theta_1\in[0,2\pi]$, we define the two initial planes intersecting on the $y$-axis as
\[
\widetilde p_1(x,y,z) = z, \hspace{1cm} \widetilde p_2(x,y,z;\theta_1) = -\sin(\theta_1)x + \cos(\theta_1)z,
\]
with normal vectors $\widetilde n_1 = [0,0,1]^T$ and $\widetilde n_2 = [ -\sin(\theta_1),0, \cos(\theta_1)]^T$.

Let $\theta_2$, $\theta_3$,  $\theta_4\in[0,2\pi]$ be the angles of rotation around the three Cartesian axes. The three corresponding rotation matrices are:
\[
R_x(\theta_2) = \begin{bmatrix}
1 & 0 & 0\\
0 & \cos(\theta_2) & -\sin(\theta_2)\\
0 & \sin(\theta_2) & \cos(\theta_2)
\end{bmatrix},
\hspace{1cm}
R_y(\theta_3) = \begin{bmatrix}
\cos(\theta_3) & 0 & \sin(\theta_3)\\
0 & 1 & 0\\
-\sin(\theta_3) & 0 & \cos(\theta_3)
\end{bmatrix},
\]
\[
R_z(\theta_4) = \begin{bmatrix}
\cos(\theta_4) & -\sin(\theta_4) & 0\\
\sin(\theta_4) & \cos(\theta_4) & 0\\
0&0&1
\end{bmatrix},
\]
whereas the general three-dimensional rotation matrix can be defined as $R=R_x(\theta_2)R_y(\theta_3)R_z(\theta_4)$. Such a matrix is used to rotate $\widetilde p_1$ and $\widetilde p_2$ by multiplying their normal vectors. We thus define $n_1$ and $n_2$ as $n_1=R\widetilde n_1$ and $n_2=R\widetilde n_2$. Note that the origin belongs to both rotated planes
$
\widehat p_i(x,y,z) = n_i^T\begin{bmatrix} x\\y\\z \end{bmatrix}$, $ i=1,2$.
To obtain the sought planes $p_1(x,y,z;\theta) $ and $p_2(x,y,z;\theta) $, we apply a last translation map. Since translations in the direction of the intersection line between the two planes would not affect the planes position, the direction of the applied translation lies in the plane orthogonal to such direction. The admissible translation vectors are thus parametrized in polar coordinates in such plane, where $\theta_5\in[0,2\pi]$ represents the angle and $\theta_6\in[0,1]$ the scaling factor of the translation distance (i.e. when $\theta_6=0$ no translation is applied, when $\theta_6=1$ the translation is such that a single point of the line belonging to both planes lies in the central cell).

Therefore, the parameter $\theta\in\R^6$ describing the two final planes can vary in $[0,2\pi]^5\times [0,1]$. Once the two planes $p_1(x,y,z;\theta) $ and $p_2(x,y,z;\theta) $ are known, the region A can be easily defined as the set of points $(x,y,z)\in\R^3$ such that $p_1(x,y,z;\theta) <0$ and $p_2(x,y,z;\theta) <0$.

Examples of distributions that can be generated with this approach are represented in Figure \ref{fig:2planes}.

\item The region A is the intersection of three half-spaces. The idea is similar to the previous case, we define three planes
\[
p_i(x,y,z;\theta) = a_i(\theta)x + b_i(\theta)y + c_i(\theta)z - d_i(\theta), \hspace{1cm} i=1,2,3,
\]
with $\theta\in\R^9$. The first 6 parameters are chosen as in the previous case for $p_1$ and $p_2$, whereas the remaining three are chosen as in the case with one plane, but adapting the definition of $d_3$ to ensure the intersection of the three planes exists and lies in the central cell. In particular, we change the sign of $d_3$ if the region $p_3(x,y,z;\theta) <0$ does not intersects the region where $p_1(x,y,z;\theta) <0$ and $p_2(x,y,z;\theta) <0$. Therefore, $\theta\in\R^0$ can vary in $[0,2\pi]^5\times [0,1]\times [0,2\pi]\times[-1,1]\times[0,1]$ and the region A is the set of points $(x,y,z)\in\R^3$ such that $p_1(x,y,z;\theta) <0$, $p_2(x,y,z;\theta) <0$ and $p_3(x,y,z;\theta) <0$.

Examples of distributions that can be generated with this approach are represented in Figure \ref{fig:3planes}.
\item The region A is the interior part of an ellipsoid. Let us consider the following parametric ellipsoid:
\[
s(x,y,z;\theta) = \frac{(x-x_0(\theta))^2}{a(\theta)^2} + \frac{(y-y_0(\theta))^2}{b(\theta)^2} + \frac{(z-z_0(\theta))^2}{c(\theta)^2} - 1,
\]
where $\theta\in\R^6$. Once more, a naive parametrization may lead to ellipsoids outside the central cell, we thus propose the following parametrization.

The center $[x_0(\theta), y_0(\theta), z_0(\theta)]$ is chosen as
\[
[x_0(\theta), y_0(\theta), z_0(\theta)] = [\theta_1, \theta_2, \theta_3] \in [-\sqrt{3(m+0.5)}, \sqrt{3(m+0.5)}]^3
\]
where $[-m-0.5, m+0.5]^3$ is the region covered by the stencil (we remind that $m$ is the stencil margin).

Let us now consider the values $a(\theta)$, $b(\theta)$ and $c(\theta)$. Three temporary values $[\tilde a(\theta),\tilde b(\theta),\tilde c(\theta)]$ are selected on the unit sphere as in Remark \ref{rem:sphere_sampl} sampling the parameters $\theta_4\in[0,2\pi]$ and $\theta_5\in[-1,1]$. Moreover, the last parameter $\theta_6$ can vary in $[0,1]$ and is used as a nonlinear scaling to compute $[ a(\theta), b(\theta), c(\theta)]$ from the values $[\tilde a(\theta),\tilde b(\theta),\tilde c(\theta)]$. The value $\theta_6=0$ is associated with a scaling such that a single point (or more generally, a finite number of points) of the ellipsoid surface intersects the border of the central cell. The value $\theta_6=1$ is associated with a scaling such that the entire central cell is inside the ellipsoid and the ellipsoid volume is minimal.

Therefore, $\theta\in\R^6$ can vary in $[-\sqrt{3(m+0.5)}, \sqrt{3(m+0.5)}]^3\times [0,2\pi]\times[-1,1]\times[0,1]$ and the region A is defined as the set of points $(x,y,z)\in\R^3$ such that $s(x,y,z;\theta)<0$.

Examples of distributions that can be generated with this approach are represented in Figure \ref{fig:ellips}. Here, the surface of the ellipsoid is represented by the vertices of the polyhedron used to approximately compute its volume. Note that, to obtain a similar integration precision for any ellipsoid, more points are needed for surfaces with larger curvature. This is automatically handled by the fact that we approximate a sphere using the points of the Fibonacci spiral \cite{keinert2015spherical}, and then we map the surface of the sphere to the surface of the ellipsoid through an affine transformation (see Section \ref{sec:dataset}).
\end{itemize}

\section{Hybridization of the VOF flux}\label{sec:hybrid-vof}


One of the advantages of neural networks is that, after an usually expensive training phase, evaluating the trained neural network is extremely efficient. It is thus possible to train a model on a large dataset just once, and then use the trained network multiple times, without the need to retrain. However, evaluating a neural network is still more expensive than some simple and less accurate VOF schemes. Therefore, we propose a hybrid strategy in which the VOF-ML scheme is used close to the interface, where high accuracy is required, and a less expensive scheme is adopted far from it, where the computation of the flux is trivial.

Observe that, if a volume fraction is exactly equal to 1, it means that inside the corresponding cell there is only fluid A. Therefore, the entirety of the flux going out from such a cell will be composed of fluid A, i.e. the flux will be 1 too. Analogously, if a volume fraction is exactly equal to 0, also the flux fraction going out from the cell will be 0. We thus fix a small threshold $\varepsilon_\text{mark}>0$ to mark the cells that are far enough from these two trivial cases.  In Section \ref{sec:numerical_results}, we set $\varepsilon_\text{mark} =0.01$, but this value may be chosen according to the used mesh size.

At the beginning of each iteration, we mark all the cells where we want to compute the flux going out from the cell with the VOF-ML scheme. All the remaining fluxes are computed with the Limited Downwind scheme, which is cheaper but less precise. For each time step, the marking rule is
\[
\text{if }\,\,\varepsilon_\text{mark} \le \alpha_{i,j,k}^n \le 1-\varepsilon_\text{mark} \,\, \implies \,\, \text{the cell }C_{i,j,k}\text{ is marked}.
\]
Such a marking procedure can be applied to each intermediate step of \eqref{eq:fv_split} to further enhance the efficiency of the method by changing the adopted flux computation scheme more often. The marked cells are named \textit{mixed cells}, whereas the unmarked cells are named \textit{pure cells}.

\end{document}

%% file: ex_shared.tex
\usepackage{lipsum}
\usepackage{amsfonts}
\usepackage{graphicx}
\usepackage{epstopdf}
\usepackage{algorithmic}
\ifpdf
  \DeclareGraphicsExtensions{.eps,.pdf,.png,.jpg}
\else
  \DeclareGraphicsExtensions{.eps}
\fi

\title{A 3D Machine Learning based Volume Of Fluid scheme without explicit interface reconstruction}

\author{Moreno Pintore\thanks{Laboratoire Jacques-Louis Lions, LJLL, Sorbonne Universit\'e; MEGAVOLT Team, INRIA; F-75005 Paris, France, (moreno.pintore@sorbonne-universite.fr).}
\and Bruno Despr\'es\thanks{Laboratoire Jacques-Louis Lions, LJLL, Sorbonne Universit\'e; Universit\'e Paris Cit\'e; CNRS; MEGAVOLT Team, INRIA; F-75005 Paris, France,
  (bruno.despres@sorbonne-universite.fr).}
}

\usepackage{amsopn}
